%% file: main.tex
\documentclass[11pt]{article}
\input{packages.tex}

\input{macros.tex}

\title{The Neighbor Graph of Linear Complementary Dual (LCD) Codes}
\author{Javier de la Cruz\footnote{The authors were supported by a Leading House for the Latin American Region Research Partnership Grant, administered by the University of St.Gallen and mandated by the Swiss State Secretariat for Education, Research and Innovation (SERI), Switzerland.}, Anna-Lena Horlemann$^*$, Marc Newman$^*$, Carlos Vela Cabello$^*$, Wolfgang Willems}

\begin{document}
\maketitle


\begin{abstract}
	Linear complementary dual (LCD) codes form an important class of linear codes with applications in cryptography, classical error correction, and quantum coding theory.
	In this paper, we study the neighbor relation on LCD codes over finite fields and the graph induced by this relation, where two codes are adjacent whenever they intersect in codimension one.
We determine the number of neighbors of an LCD code that are also LCD, and we use this result to analyze the structure of the corresponding neighbor graph.
	In particular, we prove its regularity over arbitrary finite fields and establish further regularity properties for its main structural subgraphs in the binary and odd-characteristic cases.
These results provide a graph-theoretic framework for the study of LCD codes and reveal a strong combinatorial regularity in their neighborhood structure.
\end{abstract}


\section{Introduction}

Error-correcting codes are fundamental in ensuring reliable data transmission and secure communication.
Among them, \emph{linear complementary dual (LCD) codes}, first introduced by Massey in 1992, have attracted significant interest due to their unique algebraic structure and wide range of applications~\cite{Mas92}.
LCD codes are defined as linear codes whose intersection with their dual is trivial, or---alternatively---that the vector space generated by the code and its dual is the whole ambient space.
This property makes LCD codes particularly robust against certain cryptographic attacks and enhances their effectiveness in error correction~\cite{CG14}.

Historically, LCD codes were first studied in the context of cyclic codes, where they were shown to have efficient algebraic constructions over finite fields.
In~\cite{dlCW18} the authors generalized previous results on LCD cyclic codes, offering a unified framework for the study and characterization of LCD group codes as ideals of the group algebra $\F_q G$.
At the same time, the algebraic structure and classification of LCD codes have attracted considerable attention, including the study of families with prescribed automorphism groups~\cite{BouyuklievaDeLaCruz2022}.
Moreover, the construction and classification of MDS LCD (maximum distance separable LCD codes) codes have become an active area of research due to their optimal trade-off between redundancy and error-correcting capability~\cite{8319441,9447685}.

One of the more recent applications of LCD codes is quantum error correction.
Quantum codes require special algebraic structures to protect quantum information from decoherence and noise.
While self-orthogonal codes are used in certain quantum error-correcting code constructions, LCD codes have emerged as a powerful tool in the design of entanglement-assisted quantum error-correcting codes (EAQECCs).
These codes allow quantum states to be protected against errors, utilizing entanglement between quantum systems to improve the efficiency and performance of error correction~\cite{Ren2023, Guenda2018EAQECC,Liu2019}.

In this paper, we study the neighborhood relation of LCD codes, i.e., LCD codes that intersect in co-dimension one---meaning that they are related to one another through small, localized changes.
The concept of code neighbors has been widely used for various code classes, as it provides valuable insights into their structure and performance.
For instance, it can help identify codes with good parameters, classify codes based on shared properties, and develop adaptive decoding techniques where the decoder can dynamically switch between neighboring codes to correct errors more efficiently.
In a previous study, Dougherty~\cite{Dou22} investigated the neighbors of binary self-dual codes and described the neighbor graph as a tool to analyze both the codes and their neighbors.
Inspired by this approach, the present work focuses on LCD codes---exploring their neighbors and examining the corresponding neighbor graph to gain deeper insights into their structure and potential applications in coding theory.

The graph representation of code neighbors is constructed by treating each code as a vertex and connecting two codes with an edge if they are neighbors.
By applying tools from graph theory to this representation, we can extract valuable information about the codes' relationships.
The insights might provide a powerful tool for both constructing better codes and designing more efficient decoding techniques in the future.

Our main results are the following:
\begin{itemize}
	\item In Section~\ref{sec:neighbors}, Theorem~\ref{thm:LCD-neighbors} specifies the number of neighbors of an LCD code over $\mathbb{F}_q$ that are also LCD codes.
		This result assists in describing the regularity of these graphs.
	\item In Sections~\ref{sec:Structureq2} and~\ref{sec:Structureqodd}, we study the specific structures depending on the parity of $q$; that is, in Theorem~\ref{thm:q2Subgraphsregularity} and Corollary~\ref{cor:qOddproperties} we provide the regularity of the main structural subgraphs of the LCD neighbor graphs.
\end{itemize}

The paper is structured as follows: in Section~\ref{sec:prelims}, we first state the necessary preliminaries on LCD codes and neighbors, including some new results on code neighbors.
Then, in Section~\ref{sec:neighbors}, we derive theoretical results on neighbors of LCD codes.
In Section~\ref{sec:Structureq2} we analyze the subgraphs of the LCD neighbor graph for $q = 2$ as well as for odd $q$ in Section~\ref{sec:Structureqodd}.
Finally, in Section~\ref{sec:Conclusions} we give some conclusions and further research.


\section{Preliminaries}\label{sec:prelims}


\subsection{Coding and graph theory basics}

Let $\F_q$ be a finite field.
By $\Ang{\cdot \, , \cdot}$ we denote the Euclidean inner product on $\F_q^n$.
Following the usual language in coding theory literature, we will call a vector $v\in \F_q^n$ \emph{self-orthogonal} if $\Ang{v, v} = 0$, but we remark that these vectors are also called \emph{isotropic} with respect to the inner product.

A \emph{linear code} of length $n$ is simply a subspace of the vector space $\F_q^n$.
The \emph{dual code} $C^\perp$ of $C \subseteq \F_q^n$ is defined as
\begin{align*}
	C^\perp := \{v\in \F_q^n \mid \Ang{c, v}=0 \; \forall c\in C\}.
\end{align*}
As is common, we denote by an ${[n, k]}_q$ code a linear code in $\F_q^n$ of dimension $k$.
It follows that the dual of an ${[n, k]}_q$ code is an ${[n, n - k]}_q$ code.
Any matrix $G\in \F_q^{k\times n}$ whose rows form a basis of an ${[n, k]}_q$ code is called a \emph{generator matrix} of the code, and any matrix $H \in \F_q^{(n - k) \times n}$ whose kernel is the code is called a \emph{parity check matrix} of the code.

As an error-correcting code, linear codes are usually equipped with a distance, e.g., the Hamming, Lee, or rank distance.
For any of those distance functions $d$, the \emph{minimum distance} of the code $C$ is defined as 
\begin{align*}
	\min\{d(u, v) \mid u,v \in C, u\neq v\}.
\end{align*}
The minimum distance of a code defines its error detection and correction capabilities.
It is one of the main goals of coding theory to find upper bounds on the maximal achievable minimum distance and to find good code constructions, maximizing this value, for given $q,k$ and $n$.
While we will not determine any minimum distances in this paper, we want to remark that a neighbor graph can facilitate the search for codes with good minimum distances.
This fact is therefore one of the motivations for studying the neighbor graph of LCD codes, in particular for small field sizes.

Given a graph $G = (V, E)$ and a subset of vertices $S \subseteq V$, the \emph{induced subgraph} $G[S]$ is the graph with vertices $S$ and edges $\{(u, v) \in E : u, v \in S\}$.
Given two disjoint subsets of vertices $S, S' \subseteq V$, we define $G[S, S']$ to be the bipartite graph on independence sets $S$ and $S'$ with edges $\{(u, v) \in S \times S' : (u, v) \in E\}$.

Given a graph $G$, we will denote by $\lambda_i(G)$ the $i$th largest eigenvalue of the adjacency matrix of $G$ and define $\Spec(G) = (\lambda_1(G), \lambda_2(G), \ldots)$ to be the \emph{spectrum} of $G$, that is the list of all its eigenvalues in nonincreasing order.

A partition of the vertices of a graph $G$ is called \emph{equitable} if, for a partition $V_1, \ldots, V_\ell$, the number of edges between any vertex $v \in V_i$ and any set $V_j$ of the partition (where $i$ is not necessarily different than $j$) depends only on $i$ and $j$ and not on the choice of $v$.
	Specifically, this means that $G[V_i]$ is regular and $G[V_i, V_j]$ is biregular for all choices of $i$ and $j$.

\begin{proposition}\cite{GR13}\label{prop:equitable}
	The orbits of any graph automorphism form an equitable partition.
\end{proposition}

We denote by $\GL_n(F_q)$ the general linear group and by $O_n(q)$ the orthogonal group acting on $\F_q^n$ (i.e., the invertible and the orthogonal matrices with respect to the Euclidean inner product over $\F_q$ of size $n \times n$).
The orbit of some element $C$ under the action of a group $G$ will be denoted by $\Orb_{G}(C)$.
Moreover, we will use the notation 
\begin{align*}
	\Leg{x}{q} := \begin{cases}
		 0 &\text{if $x = 0$,} \\
		 1 &\text{if $x$ is a square in $\F_q$,} \\
		-1 &\text{else.}
	\end{cases}
\end{align*}
for the \emph{Legendre symbol} indicating if $x$ is a quadratic residue in $\F_q$.
Finally, recall that the \emph{$q$-binomial coefficient} is defined as
\begin{align*}
	\Qbinom{n}{k}{q} := \prod_{i = 1}^k \frac{q^{n - k + i} - 1}{q^i - 1}.
\end{align*}
It is well-known that this value is the number of $k$-dimensional subspaces of $\F_q^n$ (i.e., ${[n, k]}_q$ linear codes) and that $\Qbinom{n}{k}{q} = \Qbinom{n}{n - k}{q}$.

We will additionally need the following results.

\begin{proposition}\cite{DSV45}\label{prop:spec-connect}
	Let $G$ be an $r$-regular graph.
	Then $r$ will be the largest eigenvalue in $\Spec(G)$ and $G$ is connected if and only if $r$ has multiplicity 1 in $\Spec(G)$.
\end{proposition}

We are going to use the proposition above in the following way: if we can show that the regularity degree $r$ of a graph $G$ fulfills $r>\lambda_2(G)$, then $G$ is connected. 

\begin{theorem}\cite{GR13}\label{thm:interlacing}
	Let $G$ be a graph on $n$ vertices and let $G'$ be an induced subgraph of $G$ on $n'$ vertices.
	Then for $i \in \{1, 2, \ldots, n'\}$:
	\begin{align*}
		\lambda_{n - n' + i}(G) \leq \lambda_i(G') \leq \lambda_i(G).
	\end{align*}
\end{theorem}


\subsection{Neighbors of linear codes}

While the concept of code neighbor has been used for pairs of code with various types of small changes, we will use it for codes that intersect in co-dimension $1$:

\begin{definition}
	Two linear codes of length $n$ and dimension $k$ over $\F_q$ are called \emph{neighbors} if their intersection is of dimension $k-1$.
\end{definition}

\begin{proposition}\label{prop:all-neighbors}
	Let $C \subseteq \F_q^n$ be a $k$-dimensional code.
	Then $C$ has $ \frac{(q^k-1)(q^{n-k+1}-q)}{{(q-1)}^2}$ neighbors.
\end{proposition}

\begin{proof}
	$C$ has $\Qbinom{k}{k-1}{q}=\frac{q^k-1}{q-1}$ many $(k-1)$-dimensional subspaces that can function as the possible intersection spaces with a neighbor.
	Each of these smaller subspaces, say $U$, can be extended to a (distinct) $k$-dimensional subspace with a one-dimensional subspace of the quotient space $\F_q^n/U$, out of which one will lead to $C$.
	Therefore we have
	\begin{align*}
		\frac{q^{n-(k-1)}-1}{q-1}-1 = \frac{q^{n-k+1}-q}{q-1}
	\end{align*}
	many choices to extend the intersection space to a neighbor of $C$, hence overall we get 
	\begin{align*}
		\frac{(q^k-1)(q^{n-k+1}-q)}{{(q-1)}^2}
	\end{align*}
	distinct neighbors of $C$.
\end{proof}

\begin{proposition}\label{prop:dual-neighbors}
	Let $C,C'\subseteq \F_q^n$ be two linear ${[n, k]}_q$ codes.
	Then $C$ and $C'$ are neighbors if and only if $C^\perp$ and $C'^{\perp}$ are neighbors.
\end{proposition}

\begin{proof}
	This follows from simple linear algebra, as for two ${[n, k]}_q$ codes $C, C'$ we have
	\begin{align*}
		\dim(C^\perp \cap C'^\perp)
		= \dim({(C + C')}^\perp)
		= n - \dim(C + C')
		= n - (2k - \dim(C \cap C')).
	\end{align*}
\end{proof}

The neighbor graph of all ${[n, k]}_q$ codes is known as the \emph{Grassmann graph} or \emph{$q$-Johnson graph} and is denoted by $\Jq_q(n, k)$, i.e., it has as vertices all $k$-dimensional subspaces of $\F_q^n$ and two vertices are adjacent if their intersection has dimension $k - 1$.

\begin{theorem}\cite[Theorem 9.3.3]{BCN89}\label{thm:Jq-char}
	The characteristic polynomial of $\Jq_q(n, k)$ is
	\begin{align*}
		\chi(t)
		= \prod_{j = 0}^{\min(\{k, n - k\})}
		{\left(
			t - q^{j + 1} \frac{q^{k - j} - 1}{q - 1} \cdot \frac{q^{n - k - j} - 1}{q - 1} + \frac{q^j - 1}{q - 1} \right)}^{\left( \Qbinom{n}{j}{q} - \Qbinom{n}{j - 1}{q} \right)}.
	\end{align*}
	Note that the eigenvalue $\lambda_1(\Jq_q(n, k))$ is in correspondence with the root of the factor indexed by $j = 0$ and the eigenvalue $\lambda_2(\Jq_q(n, k))$ corresponds to $j = 1$.
\end{theorem}

While the following result is known, we include a proof for completeness.

\begin{theorem}\label{thm:biregular}
	Let $H \leq \GL_n(\F_q)$ and let $C$ be an ${[n, k]}_q$ linear code being acted on on the right by elements of $H$.
	Then $\Jq_q(n, k)[\Orb_H(C)]$ is regular.
	Furthermore, if another ${[n, k]}_q$ linear code $C'$ exists such that $C' \notin \Orb_H(C)$, then the bipartite graph $\Jq_q(n, k)[\Orb_H(C), \Orb_H(C')]$ is biregular.
\end{theorem}

\begin{proof}
Let $C_1,C_2$ be ${[n, k]}_q$ linear codes.
	For any $M\in \GL_n(\F_q)$ we have $\dim(C_1 \cap C_2) = \dim((C_1 \cdot M) \cap (C_2 \cdot M))$ and therefore $C_1$ and $C_2$ are neighbors if and only if $C_1 \cdot M$ and $C_2 \cdot M$ are neighbors.
	Thus, $M$ represents a graph isomorphism of $\Jq_q(n, k)$ and, by Proposition~\ref{prop:equitable}, we obtain the regularity and biregularity of the subgraphs.
\end{proof}

For the rest of the paper, we will focus specifically on the orbits of LCD codes under the action of the orthogonal group.
While similar results might be of interest with regards to other groups---e.g., isometries generated by permutation or monomial matrices---the number of orbits in these cases are highly dependent on the parameters of the codes.
However, in the case of the orthogonal group, as we will see, we always have (at most) 3 orbits for $q = 2$ and 2 orbits for odd $q$ and can therefore explicitly study these subgraphs in generality for any choice of $q$, $n$, or $k$.


\subsection{Linear complementary dual (LCD) codes}

We now compile some results on linear complementary dual (LCD) codes that we will use later.
\begin{definition}
	A linear code $C\subseteq \F_q^n$ is called \emph{linear complementary dual (LCD)} if 
	\begin{align*}
		C \cap C^\perp = \{0\},
	\end{align*}
	or equivalently if $C \oplus C^\perp = \F_q^n$.
\end{definition}

The following proceeds directly from the definition.
\begin{lemma}
	If $C$ is LCD, then each $v\in \F_q^n$ can uniquely be written as
	\begin{align*}
		v = v|_C + v|_{C^\perp}
	\end{align*}
	with $v|_C \in C$ and $v|_{C^\perp} \in C^\perp$.
\end{lemma}

The following characterization is due to Massey~\cite{Mas92}.
\begin{theorem}\label{thm:LCD-equivalence}
	Let $G$ and $H$ respectively be a generator matrix and a parity-check matrix of a code $C\subseteq \F_q^n$.
	Then the following properties are equivalent:
	\begin{enumerate}[(a)]
		\item $C$ is LCD,
		\item $C^\perp$ is LCD,
		\item $GG^\top$ is nonsingular,
		\item $HH^\top$ is nonsingular.
	\end{enumerate}
\end{theorem}

\begin{remark}
	As usual, we do not consider any proper binary extension fields, since any LCD code over a field of even characteristic has an equivalent LCD code over $\F_2$ (see Proposition~3 of~\cite{CG14full}).
	Moreover, we will distinguish two types of binary LCD codes, namely the \emph{even-like} ones, where $\sum_{i = 1}^n c_i = 0$
	for all $(c_1, \ldots, c_n) \in C$, and the \emph{odd-like} ones that do not fulfill the equation above.
\end{remark}

It is easy to see that since all vectors orthogonal to the whole code are outside of the code, the inner product $\Ang{\cdot}$ is regular (or non-degenerate) on any LCD code, i.e., if $\Ang{x,y}=0$  for all codewords $y$, then $x=0$.
The regularity gives rise to the following characterization of LCD codes.

\begin{theorem}\cite[Theorems 3,5,25]{CMTQ18}
	\begin{enumerate}[(a)]
		\item Let $C\subseteq \F_2^n$ be an odd-like binary code of dimension $k$.
			Then $C$ is LCD if and only if there exists an orthonormal basis of $C$.
		\item Let $C\subseteq \F_2^n$ be an even-like binary code of dimension $k$.
			Then, $C$ is LCD if and only if $k$ is even and there exists a basis $b_1, b'_1, \ldots, b_{\frac{k}{2}},b'_{\frac{k}{2}}$ of $C$ such that for any $i, j \in \{1, 2, \ldots, k\}$
			\begin{enumerate}[(i)]
				\item $\Ang{b_i , b_i} = \Ang{b'_i, b'_i}=0$; 
				\item $\Ang{b_i , b'_j} = 0$ if $i \neq j$;
				\item $\Ang{b_i , b'_i} = 1$.
			\end{enumerate} 
		\item Let $C$ be an ${[n, k]}_q$ code for $q$ odd.
			Then, $C$ is LCD if and only if there is a basis $b_1, \ldots, b_k$ of $C$ such that for any $i, j \in \{1, 2, \ldots, k\}$, and some $\delta \in \F_q^\times$
			\begin{enumerate}[(i)]
				\item $\Ang{b_i , b_j} = 0$ if $i \neq j$; 
				\item $\Ang{b_i , b_i} = 1$ if $i<k$;
				\item $\Ang{b_k , b_k} = \delta$.
			\end{enumerate} 
	\end{enumerate}
\end{theorem}

The LCD property of a code $C\subseteq \F_q^n$ is invariant under the action of the orthogonal group on $\F_q^n$.
This action splits the set of all ${[n, k]}_q$ LCD codes into several orbits, depending on the parameters.

\begin{theorem}\cite[Theorem 12, Proposition 28]{CMTQ18}\label{thm:orbits}
	\begin{enumerate}[(a)]
		\item Let $q = 2$ and $k$ and $n$ be two positive integers such that $k < n$.
			If they exist, denote by $\CodeCoo$ (resp.\ $\CodeCoe$) an odd-like ${[n, k]}_2$ LCD code whose dual is odd-like (resp.\ even-like), and by $\CodeCeo$ (resp.\ $C_{ee}$) an even-like ${[n, k]}_2$ LCD code whose dual is odd-like (resp.\ even-like).
			\begin{enumerate}[(i)]
				\item If $n$ is odd and $k$ is odd, then the set of ${[n, k]}_2$ LCD codes can be decomposed as the following disjoint union of orbits
					\begin{align*}
						\Orb_{\Ort_n(\F_2)}(\CodeCoo) \cup \Orb_{\Ort_n(\F_2)}(\CodeCoe).
					\end{align*}
				\item If $n$ is odd and $k$ is even, then the set of ${[n, k]}_2$ LCD codes can be decomposed as the following disjoint union of orbits
					\begin{align*}
						\Orb_{\Ort_n(\F_2)}(\CodeCoo) \cup \Orb_{\Ort_n(\F_2)}(\CodeCeo).
					\end{align*}
				\item If $n$ is even and $k$ is odd, then the set of ${[n, k]}_2$ LCD codes consist of only one orbit,
					\begin{align*}
						\Orb_{\Ort_n(\F_2)}(\CodeCoo).
					\end{align*}
				\item If $n$ is even and $k$ is even, then the set of ${[n, k]}_2$ LCD codes can be decomposed as the following disjoint union of orbits
					\begin{align*}
						\Orb_{\Ort_n(\F_2)}(\CodeCoo) \cup \Orb_{\Ort_n(\F_2)}(\CodeCoe) \cup \Orb_{\Ort_n(\F_2)}(\CodeCeo).
					\end{align*}
			\end{enumerate}
		\item Let $q$ be odd, $k$ and $n$ be two positive integers with $k < n$.
			Denote by $C_+$ (resp.\ $C_-$) an ${[n, k]}_q$ LCD code with $\det(G G^\top)$ being a square (resp.\ a non-square) in $\F_q$, for $G$ being any generator matrix of the code.
			Then the set of ${[n, k]}_q$ LCD codes can be decomposed as the following disjoint union of orbits 
			\begin{align*}
				\Orb_{\Ort_n(\F_q)}(C_+) \cup \Orb_{\Ort_n(\F_q)}(C_-).
			\end{align*}
	\end{enumerate}
\end{theorem}

Note that the orthogonal group keeps the type described above of the LCD codes, hence all binary codes within an orbit have the same parity-like type, and all non-binary codes within the $\Orb_{\Ort_n(\F_q)}(C_+)$ (resp.\ $\Orb_{\Ort_n(\F_q)}(C_-)$) orbit have a square (resp.\ non-square) determinant of $G G^\top$.
The exact cardinalities of all these orbits have been determined in~\cite{CMTQ18} and depend on the parities of $q$, $n$ and $k$:

\begin{theorem}\cite[Corollaries 17 and 32]{CMTQ18}
	Let $k$ and $n$ be two positive integers with $k < n$.
	\begin{enumerate}[(a)]
		\item The number of ${[n, k]}_2$ LCD codes is 
			\begin{align*}
				\mu:= \begin{cases}
					2^{\frac{nk-k^2+n-1}{2}} \Qbinom{\frac{n}{2}-1}{\frac{k-1}{2}}{4}
					& \text{if $k$ odd, $n$ even,} \\
					2^{\frac{(k+1)(n-k)}{2}} \Qbinom{\frac{n-1}{2}}{\frac{k-1}{2}}{4}
					& \text{if $k$ odd, $n$ odd,} \\
					2^{\frac{k(n-k+1)}{2}} \Qbinom{\frac{n-1}{2}}{\frac{k}{2}}{4}
					& \text{if $k$ even, $n$ odd,} \\
					2^{\frac{k(n-k)}{2}} \left( 2^{n-k} \Qbinom{\frac{n}{2}-1}{\frac{k}{2}-1}{4} + \Qbinom{\frac{n}{2}-1}{\frac{k}{2}}{4} \right)
					& \text{if $k$ even, $n$ even.}
				\end{cases}
			\end{align*}
		\item If $q$ is odd, then the number of ${[n, k]}_q$ LCD codes is 
			\begin{align*}
				\nu := \begin{cases}
					q^{\frac{k(n-k)-1}{2}} \left(q^{\frac{n}{2}} - \Leg{{(-1)}^{\frac{n}{2}}}{q} \right) \Qbinom{\frac{n}{2}-1}{\frac{k-1}{2}}{q^2}
					& \text{if $k$ odd, $n$ even,} \\
					q^{\frac{(k+1)(n-k)}{2}} \Qbinom{\frac{n-1}{2}}{\frac{k-1}{2}}{q^2}
					& \text{if $k$ odd, $n$ odd,} \\
					q^{\frac{k(n-k+1)}{2}} \Qbinom{\frac{n-1}{2}}{\frac{k}{2}}{q^2}
					& \text{if $k$ even, $n$ odd,} \\
					q^{\frac{k(n-k)}{2}} \Qbinom{\frac{n}{2}}{\frac{k}{2}}{q^2}
					& \text{if $k$ even, $n$ even.}
				\end{cases}
			\end{align*}
	\end{enumerate}
\end{theorem}

For odd $q$ we define the \emph{sign} of an ${[n,k]}_q$ LCD code $C$ with generator matrix $G\in \F_q^{k\times n}$ as 
\begin{align*}
	\Sign(C) := \Leg{\det(G G^\top)}{q}
\end{align*}
(occasionally, depending on the context, represented with some abuse of notation as $+$, $-$, or $0$).
This is independent of the choice of generator matrix, since if $G' = M G$ for some $k \times k$ invertible matrix $M$, then 
\begin{align*}
	\Leg{\det(G' G'^\top)}{q}
	&= \Leg{\det(M) \det(G G^\top) \det(M^\top)}{q} \\
	&= {\Leg{\det(M)}{q}}^2 \Leg{\det(G G^\top)}{q}
	= \Leg{\det(G G^\top)}{q}.
\end{align*}
The following proposition shows that the dual of an LCD code has the same sign as the original.

\begin{proposition}[\cite{CMTQ18}, Proposition 5.4]\label{prop:dual-sign}
	Let $q$ be odd and $C$ be an ${[n,k]}_q$ LCD code.
	Then $\Sign(C) = \Sign(C^\bot)$.
\end{proposition}


\section{Number of neighbors among LCD codes}\label{sec:neighbors}

In this section we derive the number of neighbors of LCD codes.
It turns out that this number is independent of the choice of LCD code (given fixed parameters) and therefore also implies that the corresponding neighbor graph is regular.
We first state some preliminary lemmata which will then be used to prove the first main result on the number of LCD neighbors of a given LCD code in Theorem~\ref{thm:LCD-neighbors}.

\begin{lemma}\cite[Theorems 6.26 and 6.27]{LN08}\label{lem:solutions-quadratic-form}
	Let $q$ be odd and $Q$ a nondegenerate quadratic form over $\F_q$ in $k$ indeterminates.
	\begin{enumerate}[(a)]
		\item If $k$ is even, then for $b \in \F_q$, the number of solutions  of the equation $Q(x_1, \ldots, x_k) = b$ in $\F_q^k$ is 
			\begin{align*}
				q^{k - 1} + q^{\frac{k - 2}{2}} v(b) \Leg{{(-1)}^{\frac{k}{2}} \det(Q)}{q},
			\end{align*}
			where $v(b) := -1$ for $b \in \F_q^\times$ and $v(0) := (q - 1)$.
		\item If $k$ is odd, then for $b \in \F_q$, the number of solutions  of the equation $Q(x_1, \ldots, x_k) = b$ in $\F_q^k$ is 
			\begin{align*}
				q^{k - 1} + q^{\frac{k - 1}{2}} \Leg{{(-1)}^{\frac{k - 1}{2}} b \det(Q)}{q}.
			\end{align*}
	\end{enumerate}
\end{lemma}

\begin{lemma}\label{lem:self-orthogonal}
	Let $U\subseteq \F_q^n$ be a $k$-dimensional LCD code with generator matrix $G\in \F_q^{k\times n}$.
	\begin{enumerate}[(a)]
		\item If $q$ is even, then either $U$ contains only self-orthogonal vectors, or it contains $q^{k-1}$ many self-orthogonal vectors (including the zero vector).
		\item If $q$ is odd, then $U$ contains 
			\begin{align*}
				\begin{cases}
					q^{k - 1}
					& \text{if $k$ is odd} \\
					q^{k - 1} + \frac{q - 1}{q} \Leg{\det(Q)}{q} {\Leg{-1}{q} }^{\frac{k}{2}} q^{\frac{k}{2}}
					& \text{if $k$ is even}
				\end{cases}
			\end{align*}
			many self-orthogonal vectors (including the zero vector) where $Q := G G^\top$ represents the quadratic form in $k$ variables describing the subset of self-orthogonal vectors in $U$.
	\end{enumerate}
\end{lemma}

\begin{proof}
	\begin{enumerate}[(a)]
		\item In even characteristic, squaring is a linear function and so the self-orthogonal vectors form a hyperplane in $U$ given by the normal vector being the all-one vector, as long as the all-one vector is not orthogonal to $U$ itself.
		\item In odd characteristic the subset of self-orthogonal vectors form a quadratic variety, given by a non-degenerate\footnote{Since $U$ is LCD it has a basis of non-self-orthogonal vectors and hence $Q$ can be assumed to be non-degenerate.} quadratic form $Q$ in $k$ variables (since $U$ has dimension~$k$).
			Using Lemma~\ref{lem:solutions-quadratic-form} the statement follows.
			\qedhere{}
	\end{enumerate}
\end{proof}

In the following we denote by $\Proj(\F_q^k)$ the projective space of dimension $k-1$ over $\F_q$.

\begin{lemma}\label{lem:character-sum}
	Let $q$ be odd and $Q$ be a non-degenerate quadratic form in $k$ indeterminates.
	Then 
	\begin{align*}
		\sum_{x\in \Proj(\F_q^k)} \Leg{Q(x)}{q}
		= \begin{cases}
			{\Leg{-1}{q}}^{\frac{k - 1}{2}} \Leg{\det(Q)}{q} q^{\frac{k - 1}{2}}
			 & \text{if $k$ is odd}, \\
			 0
			 & \text{otherwise}.
		\end{cases}
	\end{align*}
\end{lemma}

\begin{proof}
	For the following, let
	\begin{align*}
		\F_q^+ := \left\{x \in \F_q \middle| \Leg{x}{q} = +1 \right\}, \\
		\F_q^- := \left\{x \in \F_q \middle| \Leg{x}{q} = -1 \right\}.
	\end{align*}

	We first show that the sum is actually well defined, i.e., independent of the chosen representative of the projective points.
	For this note that, for any $\lambda \in \F_q^\times$,
	\begin{align*}
		\Leg{Q(\lambda x)}{q}
		= \Leg{\lambda^2 Q(x)}{q}
		= \Leg{\lambda^2 }{q} \Leg{ Q(x)}{q}
		= \Leg{Q(x)}{q}.
	\end{align*}
	Next, we transform the sum over projective points into a sum over the affine vector space, using that every equivalence class of a projective point contains $q-1$ vectors which all results in the same Legendre symbol under a quadratic form, as shown above (we may choose to include the zero vector or not):
	\begin{align*}
		\sum_{x \in \Proj(\F_q^k)} \Leg{Q(x)}{q}
		= \frac{1}{q - 1} \sum_{x \in \F_q^k} \Leg{Q(x)}{q}
	\end{align*}
	The latter sum is known to be equal to zero if $k$ is even and equal to $(q - 1) {\Leg{-1}{q}}^{\frac{k - 1}{2}} q^{\frac{k - 1}{2}}$ if $k$ is odd.
	This last part follows from Lemma~\ref{lem:solutions-quadratic-form}, which implies that the number of $x$ such that $Q(x)$ is a square and the number of $x$ such that $Q(x)$ is non-square is equal if $k$ is even, and differs by $(q - 1) {\Leg{-1}{q}}^{\frac{k - 1}{2}} q^{\frac{k - 1}{2}}$ if $k$ is odd.
	We illustrate the more difficult case of odd $k$ in the following:
	\begin{align*}
		   \sum_{x\in \F_q^k} \Leg{Q(x)}{q}
		&= \sum_{b \in \F_q^\times} \sum_{x \in Q^{-1}(b)} \Leg{Q(x)}{q} \\
		&= \sum_{b \in \F_q^+} \sum_{x \in Q^{-1}(b)} \Leg{Q(x)}{q}
		 + \sum_{b \in \F_q^-} \sum_{x \in Q^{-1}(b)} \Leg{Q(x)}{q} \\
		&= \sum_{b \in \F_q^+} \left( q^{k - 1} + q^{\frac{k - 1}{2}} \Leg{{(-1)}^{\frac{k-1}{2}} b \det(Q)}{q} \right) \\
		&\quad - \sum_{b \in \F_q^-} \left( q^{k - 1} + q^{\frac{k - 1}{2}} \Leg{{(-1)}^{\frac{k - 1}{2}} b \det(Q)}{q} \right) \hspace{2cm} \text{ by Lemma~\ref{lem:solutions-quadratic-form}}\\
		&=q^{k-1}\underbrace{\left(|\F_q^+|- |\F_q^-| \right)}_0 +  q^{\frac{k - 1}{2}} \Leg{{(-1)}^{\frac{k - 1}{2}} \det(Q)}{q} \left(\sum_{b \in \F_q^+} \Leg{b}{q}- \sum_{b \in \F_q^-} \Leg{b}{q}\right) \\
		&= q^{\frac{k - 1}{2}} \Leg{{(-1)}^{\frac{k - 1}{2}} \det(Q)}{q} \sum_{b \in \F_q^\times} 1 = (q - 1) {\Leg{-1}{q}}^{\frac{k - 1}{2}} \Leg{\det(Q)}{q} q^{\frac{k - 1}{2}}.
	\end{align*}
\end{proof}

\begin{lemma}\label{lem:character-sum-odd-code}
	Let $q$ and $k$ be odd and let $C \subseteq \F_q^n$ be a $k$-dimensional LCD code with generator matrix $G$.
	Then
	\begin{align*}
		\sum_{b \in C} \Leg{\Ang{b, b}}{q}
		= \Leg{\det(G G^\top)}{q} \sum_{x \in \F_q^k} \Leg{\Ang{x, x}}{q}.
	\end{align*}
\end{lemma}

\begin{proof}
	Denote by $e_i$ the $i$th unit vector in $\F_q^n$.
	Choose $a, b \in \F_q^\times$ such that $\gamma = a^2 + b^2$ is a non-square in $\F_q$ and define the matrices
	\begin{align*}
		G_+ := \begin{bmatrix}
			e_1 \\
			e_2 \\
			\vdots \\
			e_{k - 1} \\
			e_k
		\end{bmatrix},
		\qquad
		G_- := \begin{bmatrix}
			e_1 \\
			e_2 \\
			\vdots \\
			e_{k - 1} \\
			a \cdot e_k + b \cdot e_{k + 1}
		\end{bmatrix}.
	\end{align*}
	Let $C_+$ be the LCD code generated by $G_+$ and $C_-$ be the LCD code generated by $G_-$.

	If $C \in \Orb_{\Ort_n(\F_q)}(C_+)$, there exists some $O \in \Ort_n(\F_q)$ such that $G = G_+ \cdot O$ (see Theorem~\ref{thm:orbits}) and hence
	\begin{align*}
		\sum_{b \in C} \Leg{\Ang{b, b}}{q}
		&= \sum_{x \in \F_q^k} \Leg{x \cdot G \cdot G^\top \cdot x^\top}{q} \\
		&= \sum_{x \in \F_q^k} \Leg{x \cdot G_+ \cdot O \cdot O^\top \cdot G_+^\top \cdot x^\top}{q} \\
		&= \sum_{x \in \F_q^k} \Leg{x \cdot G_+ \cdot G_+^\top \cdot x^\top}{q} \\
		&= \sum_{x \in \F_q^k} \Leg{x \cdot I_k \cdot x^\top}{q} \\
		&= \sum_{x \in \F_q^k} \Leg{\Ang{x, x}}{q} .
	\end{align*}
	The statement follows with the fact that $\det(G \cdot G^\top) = \det(G_+ \cdot O \cdot O^\top \cdot G_+^\top) = \det(I_k) = 1$.

	If $C \in \Orb_{\Ort_n(\F_q)}(C_-)$, there exists some $O \in \Ort_n(\F_q)$ such that $G = G_- \cdot O$ (see Theorem~\ref{thm:orbits}) and hence---similarly to above--- 
	\begin{align*}
		\sum_{b \in C} \Leg{\Ang{b, b}}{q}
		&= \left( \sum_{x \in \F_q^{k - 1}} \Leg{\Ang{x, x}}{q} \right) + \left( \sum_{\lambda \in \F_q^\times} \sum_{x \in \F_q^{k - 1}} \Leg{\lambda^2 (\gamma + \Ang{x, x})}{q} \right) \\
		&= \sum_{\lambda \in \F_q^\times} \sum_{x \in \F_q^{k - 1}} \Leg{\lambda^2}{q} \Leg{\gamma + \Ang{x, x}}{q} \\
		&= (q - 1) \sum_{x \in \F_q^{k - 1}} \Leg{\gamma + \Ang{x, x}}{q}
	\end{align*}
	where the second equality follows from Lemma~\ref{lem:character-sum}.
	By Lemma~\ref{lem:solutions-quadratic-form}, considering the Euclidean inner product as a quadratic form, and with 
	\begin{align*}
		\sigma_b &:= q^{k - 2} - q^{\frac{k - 3}{2}} \Leg{{(-1)}^{\frac{k - 1}{2}}}{q}, \\
		\sigma_0 &:= q^{k - 2} + (q - 1) q^{\frac{k - 3}{2}} \Leg{{(-1)}^{\frac{k - 1}{2}}}{q}
	\end{align*}
	being the number of $x\in\F_q^{k-1}$ such that $\Ang{x, x}=b \in \F_q^\times$, or such that $\Ang{x, x}=0$, respectively.
	Note that if $\Ang{x, x} = 0$, then $\Leg{\gamma + \Ang{x, x}}{q} = -1$; and if $\Ang{x, x} = -\gamma$, then $\Leg{\gamma + \Ang{x, x}}{q} = 0$.
	For all other values of $x$, because we have an equal number of nonzero squares and non-squares, we map to equal numbers of the remaining $\frac{q - 1}{2}$ squares and $\frac{q - 3}{2}$ non-squares in $\F_q^\times$.
	Finishing up:
	\begin{align*}
		\sum_{b \in C} \Leg{\Ang{b, b}}{q}
		&= (q - 1) \sum_{x \in \F_q^{k - 1}} \Leg{\gamma + \Ang{x, x}}{q} \\
		&= (q - 1) (-\sigma_0 + \sigma_b (\frac{q - 1}{2} - \frac{q - 3}{2})) \\
		&= (q - 1) (\sigma_b - \sigma_0) \\
		&= (q - 1) \left(-q \cdot q^{\frac{k - 3}{2}} \Leg{{(-1)}^{\frac{k - 1}{2}}}{q} \right) \\
		&= - (q - 1) \left(q^{\frac{k - 1}{2}} \Leg{{(-1)}^{\frac{k - 1}{2}}}{q} \right), \\
		&= - \sum_{x \in \F_q^k} \Leg{\Ang{x, x}}{q}
	\end{align*}
	where the final equality comes from Lemma~\ref{lem:character-sum}.
	The final statement follows with the fact that $\det(G \cdot G^\top) = \det(G_- \cdot O \cdot O^\top \cdot G_-)= \det(G_- \cdot  G_-^\top) = \gamma$, which is a non-square.
\end{proof}

\begin{theorem}\label{thm:LCD-neighbors}
	Let $C \subseteq \F_q^n$ be a $k$-dimensional LCD code.
	The number of LCD neighbors of $C$ is
	\begin{align*}
		\begin{cases}
			\frac{q - 1}{q} N - {\Leg{-1}{q}}^{\frac{n}{2}} q^{\frac{n}{2} - 1}
			& \text{if $q$ is odd, $n$ is even, and $k$ is odd,} \\
			\frac{q - 1}{q} N
			& \text{otherwise,}
		\end{cases}
	\end{align*}
	where
	\begin{align*}
		N := \frac{(q^k - 1)(q^{n - k + 1} - q)}{{(q - 1)}^2}
	\end{align*}
	is the total number of neighbors of $C$.
\end{theorem}

\begin{proof}
	Throughout this proof we denote by $G_U$ a generator matrix of some subspace $U$, and by $Q_U:=G_U \cdot G_U^\top$ the representation of the corresponding quadratic form.
	As before, $C$ has $\frac{q^k - 1}{q - 1}$ many $(k - 1)$-dimensional subspaces that can function as the possible intersection with a neighbor.
	Each of these smaller subspaces, say $U$, can be extended to a (distinct) $k$-dimensional subspace with a one-dimensional subspace of the quotient space $\F_q^n / U$, out of which one will lead to $C$.
	Now we distinguish two cases:
	\begin{itemize}
		\item \textbf{$\mathbf{U}$ is LCD:}
			In this case, we get that $U^\perp \cong \F_q^n / U$ and only a self-orthogonal one-dimensional subspace will lead to a non-LCD neighbor.
			Vice-versa, the non-self-orthogonal projective points lead to LCD codes.
			Applying Lemma~\ref{lem:self-orthogonal}, we get
			\begin{align*}
				\begin{cases}
						\frac{q^{n - k + 1} - 1}{q - 1}
					  - \frac{q^{n - k} - 1}{q - 1} - 1
					  = q^{n - k} - 1
					  &\text{if $q$ even or $n - k$ even}, \\
						q^{n - k} - 1 - \Leg{\det(Q_U)}{q}
						{\Leg{-1}{q}}^{(n - k + 1) / 2}
						q^{(n - k - 1) / 2}
					  &\text{otherwise},
				   \end{cases}
			\end{align*}
			many cosets which give rise to an LCD neighbor where $Q_U$ denotes the quadratic form in $n - k + 1$ variables describing the self-orthogonal vectors in $U^\perp$.
			Here the second cardinality above derives from the number of points in $U^\perp$ minus the number of self-orthogonal points and minus one for the coset leading back to $C$, that is,
			\begin{align*}
				&\frac{q^{n - k + 1} -1}{q - 1}
				- \frac{%
					q^{n - k} + \frac{q - 1}{q}
					\Leg{\det(Q_U)}{q}
					{\Leg{-1}{q}}^{(n - k + 1)/2}
					q^{(n - k + 1) / 2} - 1
				}{q - 1} - 1.
			\end{align*}
		  \item \textbf{$\mathbf{U}$ is not LCD:}
			In this case all cosets corresponding to an element in $U^\perp \cap (\F_q^n / U)$ (note that this intersection must have dimension $n - k$) lead to a non-LCD neighbor.
			Hence we have
			\begin{align*}
				\frac{q^{n - k + 1} - 1}{q - 1} - \frac{q^{n - k} - 1}{q - 1} - 1
				= q^{n - k} - 1
			\end{align*}
			many cosets that give rise to an LCD neighbor.
	\end{itemize}
	If $q$ even or $n - k$ even, we therefore have
	\begin{align*}
		\frac{(q^k - 1) (q^{n - k} - 1)}{q - 1} = \frac{q - 1}{q} N
	\end{align*}
	distinct LCD neighbors of $C$.
	In the case that both $q$ and $n - k$ are odd, for the case that $n$ is odd and $k$ is even, by Proposition~\ref{prop:dual-neighbors}, we can take the dual of $C$ and consider the case with $n$ odd and $k$ odd (and $n - k$ even) which is already proven.

	Therefore, we have only the case remaining where $q$ is odd, $n$ is even, and $k$ is odd.
	Note that in this case, if $G$ is a generator matrix for $C$ and $H$ is a parity check matrix for $C$, we then have that $U^\perp$ is generated by $G_{U^\perp}=\begin{bmatrix} H \\ b \end{bmatrix}$ for some $b \in C$ because $C$ is LCD\@.
	Furthermore, up to scalar multiples, the subcodes obtained in this way are unique.
	We get that $bH^\top = 0$ and thus, with Proposition~\ref{prop:dual-sign}, 
	\begin{align*}
		\Leg{\det(Q_U)}{q}
		= \Leg{\det(Q_{U^\perp})}{q}
		&= \Leg{\det(H \cdot H^\top) \cdot \Ang{b, b}}{q} \\
		&= \Leg{\det(G \cdot G^\top)}{q} \Leg{\Ang{b, b}}{q}.
	\end{align*}

	It then follows that
	\begin{align*}
		\sum_{\substack{
			U \leq C, \\
			\dim(U) = k - 1
		}} \Leg{\det(Q_U)}{q}
		&= \frac{1}{q - 1} \sum_{b \in C} \Leg{\det(Q_U)}{q} \\
		&= \frac{1}{q - 1} \Leg{\det(G \cdot G^\top)}{q} \sum_{b \in C} \Leg{\Ang{b, b}}{q} \\
		&= \frac{1}{q - 1} {\Leg{\det(G \cdot G^\top)}{q}}^2
		\sum_{x \in \F_q^k} \Leg{\Ang{x, x}}{q} \hspace{1.1cm} \text{by Lemma~\ref{lem:character-sum-odd-code}}\\
		&= \frac{1}{q - 1} \sum_{x \in \F_q^k} \Leg{\Ang{x, x}}{q} \\
		&= \sum_{x \in \Proj(\F_q^k)} \Leg{\Ang{x, x}}{q} \\
		&= {\Leg{-1}{q}}^{\frac{k - 1}{2}} q^{\frac{k - 1}{2}}\hspace{4.8cm} \text{by Lemma~\ref{lem:character-sum} }.
	\end{align*}

	Now, counting the number of LCD neighbors, we get:
	\begin{align*}
		\sum_{\substack{
			U \leq C, \\
			\dim(U) = k - 1, \\
			\text{$U$ not LCD}
		}}
		&(q^{n - k} - 1)
		+ \sum_{\substack{
			U \leq C, \\
			\dim(U) = k - 1, \\
			\text{$U$ LCD}
		}}
		\left(
			q^{n - k} - 1 -
			\Leg{\det(Q_U)}{q}
			{\Leg{-1}{q}}^{\frac{n - k + 1}{2}}
			q^{\frac{n - k - 1}{2}}
		\right) \\
		&= \frac{(q^k - 1)(q^{n - k} - 1)}{q - 1}	 
		- \sum_{\substack{
			U \leq C, \\
			\dim(U) = k - 1, \\
			\text{$U$ LCD}
		}}
		\Leg{\det(Q_U)}{q}
		{\Leg{-1}{q}}^{\frac{n - k + 1}{2}}
		q^{\frac{n - k - 1}{2}} \\
		&= \frac{(q^k - 1) (q^{n - k} - 1)}{q - 1}	 
		- \sum_{\substack{
			U \leq C, \\
			\dim(U) = k - 1
		}}
		\Leg{\det(Q_U)}{q}
		{\Leg{-1}{q}}^{\frac{n - k + 1}{2}}
		q^{\frac{n - k - 1}{2}} \\
		&= \frac{(q^k - 1) (q^{n - k} - 1)}{q - 1}	 
		- {\Leg{-1}{q}}^{\frac{k - 1}{2}}
		q^{\frac{k - 1}{2}}
		{\Leg{-1}{q}}^{\frac{n - k + 1}{2}}
		q^{\frac{n - k - 1}{2}} \\
		&= \frac{(q^k - 1) (q^{n - k} - 1)}{q - 1}	 
		- {\Leg{-1}{q}}^{\frac{n}{2}}
		q^{\frac{n}{2} - 1}.
	\end{align*}
	Note that for the second equality, we can add the non-LCD codes to the sum since their corresponding determinant of $Q_U$ is equal to zero.
\end{proof}

\begin{example}
	Let $C\subseteq \F_3^6$ be the row space of
	\begin{align*}
		G = \begin{bmatrix}
			1 & 0 & 0 & 0 & 0 & 0 \\
			0 & 1 & 0 & 0 & 0 & 0 \\
			0 & 0 & 1 & 0 & 0 & 0
		\end{bmatrix},
	\end{align*}
	i.e., $C$ is an LCD code with a possible parity check matrix 
	\begin{align*}
		H = \begin{bmatrix}
			0 & 0 & 0 & 1 & 0 & 0 \\
			0 & 0 & 0 & 0 & 1 & 0 \\
			0 & 0 & 0 & 0 & 0 & 1
		\end{bmatrix}.
	\end{align*}
	There are $(3^3 - 1) / 2 = 13$ two-dimensional subspaces of $C$, out of which the ones with a weight-1 basis vectors in their reduced row echelon form are LCD (the other four subspaces are not LCD).

	As an example with an LCD subcode we choose 
	\begin{align*}
		U := \Ang{\begin{bmatrix}
			1 & 0 & 0 & 0 & 0 & 0\\
			0 & 1 & 0 & 0 & 0 & 0
		\end{bmatrix}}.
	\end{align*}
	Then 
	\begin{align*}
		U^\perp= \Ang{\begin{bmatrix}
			0 & 0 & 1 & 0 & 0 & 0\\
			0 & 0 & 0 & 1 & 0 & 0\\
			0 & 0 & 0 & 0 & 1 & 0\\
			0 & 0 & 0 & 0 & 0 & 1
		\end{bmatrix}}
	\end{align*}
	and we can choose the $(3^4 - 1) / 2 = 40$ one-dimensional subspaces as representatives of $\F_3^6/U$.
	Those of Hamming weight $3$ are self-orthogonal and lead to non-LCD neighbors of $C$:
	\begin{align*}
		&(0,0,1,1,1,0), &&(0,0,1,2,1,0), &&(0,0,1,1,2,0), &&(0,0,1,2,2,0), \\
		&(0,0,1,0,1,1), &&(0,0,1,0,1,2), &&(0,0,1,0,2,1), &&(0,0,1,0,2,2), \\
		&(0,0,0,1,1,1), &&(0,0,0,1,1,2), &&(0,0,0,1,2,1), &&(0,0,0,1,2,2), \\
		&(0,0,1,1,0,1), &&(0,0,1,2,0,1), &&(0,0,1,1,0,2), &&(0,0,1,2,0,2).
	\end{align*}
	Note that these are the $16$ corresponding projective points to the $32$ (non-zero) non-projective points that are self-orthogonal and which corresponds exactly to the number from Lemma~\ref{lem:self-orthogonal}.
	The remaining $40 - 16 - 1 = 23$ give rise to LCD neighbors.
	
	As an example with a non-LCD subcode we choose 
	\begin{align*}
		U:= \Ang{\begin{bmatrix}
			1 & 0 & 1 & 0 & 0 & 0\\
			0 & 1 & 1 & 0 & 0 & 0
		\end{bmatrix}}.
	\end{align*}
	Then 
	\begin{align*}
		U^\perp= \Ang{\begin{bmatrix}
			1 & 1 & 2 & 0 & 0 & 0\\
			0 & 0 & 0 & 1 & 0 & 0\\
			0 & 0 & 0 & 0 & 1 & 0\\
			0 & 0 & 0 & 0 & 0 & 1
		\end{bmatrix}},
	\end{align*}
	and whatever we choose as $\F_3^6 / U$ will intersect $U^\perp$ in $(0 \mid I_3)$ which contains $(3^3 - 1) / 2 = 13$ one-dimensional subspaces.
	Choosing any of those will lead to non-LCD neighbors of $C$.
	The remaining $(3^4 - 1) / 2 - 13 - 1 = 3^3 - 1 = 26$ give rise to LCD neighbors of $C$.
\end{example}

Theorem~\ref{thm:LCD-neighbors} shows that (for large field sizes) most of the neighbors of an LCD code are LCD themselves:
 
\begin{corollary}
	For $q = 2$ exactly half of the neighbors of any LCD code are themselves LCD codes.
	For growing $q$ the fraction of neighbors of an LCD code that are themselves LCD codes approaches $1$.
\end{corollary}

In the following two sections we will derive results about the LCD neighbor graph, first about the binary case, and afterwards about the odd $q$ case. 
Throughout both sections we will denote by $\Gq_q(n, k)$ the graph whose vertices are all ${[n, k]}_q$ LCD codes and whose edges represent the neighbor relation, i.e., two nodes are connected by an edge if and only if they are neighbors.
In other words, $\Gq_q(n, k) = \Jq_q(n, k)[\LCDq{n}{k}]$.


\section{Structure of binary LCD neighbor graphs}\label{sec:Structureq2}

We will now analyze the binary LCD neighbor graph in more detail. The first result follow straightforwardly from Theorem~\ref{thm:LCD-neighbors}.

\begin{corollary}\label{cor:regular_binary}
	The LCD neighbor graph $\Gq_2(n, k)$ is regular of degree 
	\begin{align*}
		r := (2^k-1)(2^{n-k}-1).
	\end{align*}
\end{corollary}

Next we show the connectedness of the graph:

\begin{theorem}
	The LCD neighbor graph $\Gq_2(n, k)$ is connected.
\end{theorem}

\begin{proof}
	We show that since $\Gq_2(n, k)$ is an induced subgraph of $\Jq_2(n, k)$ and is $r$-regular (with $r$ as in Corollary~\ref{cor:regular_binary}), by Proposition~\ref{prop:spec-connect} and by Theorem~\ref{thm:interlacing}, it must be connected if $r > \lambda_2(\Jq_2(n, k))$.
	By Theorem~\ref{thm:Jq-char}, 
	\begin{align*}
		r -  \lambda_2(\Jq_2(n, k))
		&= (2^k - 1)(2^{n - k } - 1) - 2^2 (2^{k - 1} - 1)(2^{n - k - 1} - 1) + 1 \\
		&= 2^{n - k + 1} - 2^{n - k} + 2^{k + 1} - 2^{k} - 2 \\
		&= \underbrace{2^{k}}_{\geq 2} + \underbrace{2^{n - k }}_{\geq 2} - 2
		> 0.
	\end{align*}
\end{proof}

Following the notation of~\cite{CMTQ18}, we define $\LCDoo{n}{k}$, $\LCDoe{n}{k}$, $\LCDeo{n}{k}$, and $\LCDee{n}{k}$ to be the sets of ${[n, k]}_2$ LCD codes which are (respectively) odd-like with odd-like dual, odd-like with even-like dual, even-like with odd-like dual, and even-like with even-like dual (which is necessarily empty).
We give the following method of classifying binary LCD codes into these sets:

\begin{proposition}\label{prop:ones}
	Let $C$ be an ${[n, k]}_2$ LCD code and let $w = (1, \ldots, 1) \in \F_2^n$.
	Then
	\begin{enumerate}[(a)]
		\item $C \in \LCDoo{n}{k}$ if and only if $w \notin C$ and $w \notin C^{\perp}$,
		\item $C \in \LCDoe{n}{k}$ if and only if $w \in C$ and $w \notin C^{\perp}$,
		\item $C \in \LCDeo{n}{k}$ if and only if $w \notin C$ and $w \in C^{\perp}$.
	\end{enumerate}
\end{proposition}

\begin{proof}
	A code $C$ is even-like if and only if $c$ has even weight for all $c \in C$.
	Since
	\begin{align*}
		\Wt(c) = \sum_{i = 1}^n c_i = \Ang{c, w}  \pmod{2}
	\end{align*}
	we have that $C$ is even-like if and only if $w \in C^{\perp}$.
	It then additionally follows that $\LCDee{n}{k}$ is necessarily empty since if both $C$ and $C^{\perp}$ are even, $w \in C \cap C^{\perp}$ but then $C$ cannot be LCD\@.
\end{proof}

We show that codes in $\LCDoo{n}{k}$ are always neighbors with codes in the other two orbits when those orbits are non-empty.

\begin{proposition}
	Consider the orbit decomposition of Theorem~\ref{thm:orbits}.
	\begin{enumerate}[(a)]
		\item If $n - k$ is even, then any element of $\Orb_{\Ort_n(\F_2)}(\CodeCoo)$ has a neighbor on $\Orb_{\Ort_n(\F_2)}(\CodeCoe)$.
		\item If $k$ is even, then any element of $\Orb_{\Ort_n(\F_2)}(\CodeCoo)$ has a neighbor on $\Orb_{\Ort_n(\F_2)}(\CodeCeo)$.
	\end{enumerate}
\end{proposition}

\begin{proof}
	We denote again by $e_i$ the $i$th unit vector of length $n$. 
	Note that $\MatGoo := (I_k \mid 0_{k \times (n - k)})$ generates a code in $\Orb_{\Ort_n(\F_2)}(\CodeCoo)$.
	\begin{enumerate}[(a)]
		\item If $n - k$ is even, consider the matrix 
			\begin{align*}
				\MatGoe := \begin{bmatrix}
					I_{k-1} &   & 0	  & \\
						  0 & 1 & \cdots & 1
				\end{bmatrix}
				.
			\end{align*}
			The code generated by $\MatGoe$ is on $\Orb_{\Ort_n(\F_2)}(\CodeCoo)$ and any other element on this orbit has a generator matrix of the form $\MatGoe \cdot M$, for some $M\in \Ort_n(\F_2)$. 
			The intersection $\MatGoo \cdot M\cap \MatGoe \cdot M$ is generated by $e_1 M, \ldots, e_{k-1} M$, which implies the statement.
		\item If $k$ is even, the same strategy can be used for $\Orb_{\Ort_n(\F_2)}(\CodeCoo)$ and $\Orb_{\Ort_n(\F_2)}(\CodeCeo)$ as follows: define $\CurB := {\{e_{2i} + e_{2i + 1}\}}_{i = 1, \ldots, \frac{k}{2} - 1} \cup {\{\sum_{j = 1}^{2i} e_i\}}_{i = 1, \ldots, \frac{k}{2}}$, then $\CurB \cup \{e_1\}$ is a basis for the same code as generated by $\MatGoo$ (on the odd-odd orbit), and $\CurB \cup \{e_k + e_{k + 1}\}$ is a basis of an even-odd code. Analogously to above, we can use the orbit structure to construct a neighbor on $\Orb_{\Ort_n(\F_2)}(\CodeCoo)$ for any element on $\Orb_{\Ort_n(\F_2)}(\CodeCeo)$.
			\qedhere{}
	\end{enumerate}
\end{proof}

To show that there are never neighbor relationships between the other two orbits, we will need the following lemma:

\begin{lemma}\label{lem:even-subcode}
	Let $C$ be an odd-like linear ${[n, k]}_2$ code.
	Then $C$ has a unique even-like ${[n, k - 1]}_2$ subcode.
\end{lemma}

\begin{proof}
	Let $w = (1, \ldots, 1) \in \F_2^n$.
	By Proposition~\ref{prop:ones}, $C$ is odd-like, so $w \notin C^{\perp}$.
	Then the code $\CodeCe := {(C^{\perp} \oplus \Ang{w})}^\perp$ is an even-like subcode of $C$ with dimension $k - 1$.
	Furthermore, the $2^{k - 1}$ vectors in $C \setminus \CodeCe$ are all of odd weight.
	Thus, $\CodeCe$ must be unique.
\end{proof}

\begin{proposition}
	Given any $n$ and $k$, no code $C \in \LCDoe{n}{k}$ is neighbors of any $C' \in \LCDeo{n}{k}$.
\end{proposition}

\begin{proof}
	If $n$ and $k$ are not both even, either $\LCDoe{n}{k}$ or $\LCDeo{n}{k}$ is empty and the statement is trivial.
	So assume $n$ and $k$ are both even and assume there exist some $C \in \LCDoe{n}{k}$ and $C' \in \LCDeo{n}{k}$ such that they are neighbors.
	Then their intersection must be even-like and therefore must be the unique even-like $(k - 1)$-dimensional subcode $\CodeCe \subset C$ defined in Lemma~\ref{lem:even-subcode}.
	Let $u \in \F_2^n$ be some (necessarily odd-weight) vector such that $C = \CodeCe \oplus \Ang{u}$ and let $w = (1, \ldots, 1) \in \F_2^n$.
	Then by Proposition~\ref{prop:ones}, $w \in C$ and $w \notin C^{\perp}$.
	It follows that $w = u + e$ for some even-weight vector $e \in \CodeCe$.
	But then $w$ has odd weight which is a contradiction since $n$ is even.
\end{proof}

We need the following theorems and lemmata to explicitly compute the biregularity degrees between the orbits. First we recall one of the main results of~\cite{CMTQ18}.

\begin{theorem}[\cite{CMTQ18}, Theorem 4.6]\label{thm:CMTQ-count}
	\begin{align*}
		\Abs{\LCDoo{n}{k}}
		&= \begin{cases}
			2^{\frac{n k - k^2 + n - 1}{2}}
			\Qbinom{\tfrac{n}{2} - 1}{\tfrac{k - 1}{2}}{4}
			&\text{if $k$ odd, $n$ even}, \\
			2^{\frac{(n - k)(k - 1)}{2}}
			(2^{n - k} - 1)
			\Qbinom{\tfrac{n - 1}{2}}{\tfrac{k - 1}{2}}{4}
			&\text{if $k$ odd, $n$ odd}, \\
			2^{\frac{k(n - k - 1)}{2}}
			(2^k - 1)
			\Qbinom{\tfrac{n - 1}{2}}{\tfrac{k}{2}}{4}
			&\text{if $k$ even, $n$ odd}, \\
			2^{\frac{k(n - k)}{2}}
			(2^k - 1)
			\Qbinom{\tfrac{n}{2} - 1}{\tfrac{k}{2}}{4}
			&\text{if $k$ even, $n$ even}.
		\end{cases} \\
		\Abs{\LCDoe{n}{k}}
		&= \begin{cases}
			2^{\frac{(n - k)(k - 1)}{2}}
			\Qbinom{\tfrac{n - 1}{2}}{\tfrac{k - 1}{2}}{4}
			&\text{if $k$ odd, $n$ odd}, \\
			2^{\frac{k(n - k)}{2}}
			\Qbinom{\tfrac{n}{2} - 1}{\tfrac{k}{2} - 1}{4}
			&\text{if $k$ even, $n$ even}, \\
			0
			&\text{otherwise}.
		\end{cases} \end{align*}
		\begin{align*}
		\Abs{\LCDeo{n}{k}}
		&= \begin{cases}
			2^{\frac{k(n - k - 1)}{2}}
			\Qbinom{\tfrac{n - 1}{2}}{\tfrac{k}{2}}{4}
			&\text{if $k$ even, $n$ odd}, \\
			2^{\frac{k(n - k)}{2}}
			\Qbinom{\tfrac{n}{2} - 1}{\tfrac{k}{2}}{4}
			&\text{if $k$ even, $n$ even}, \\
			0
			&\text{otherwise}.
		\end{cases}
	\end{align*}
\end{theorem}

\begin{lemma}\label{lem:ooeo-nec}
	If $C \in \LCDoo{n}{k}$ has a neighbor $C' \in \LCDeo{n}{k}$, then $C' = \CodeCe \oplus \Ang{u}$ where $\CodeCe$ is the unique $(k - 1)$-dimensional even-like subcode of $C$ and $u = v + d$ for some odd-weight $d \in C^{\perp}$ and $v \in C$ such that $C = \CodeCe \oplus \Ang{v}$.
\end{lemma}

\begin{proof}
	By the definition of a neighbor, the intersection $C \cap C'$ must have dimension $k - 1$ and since it is contained in an even-like code, it will also be even-like.
	It therefore must be $\CodeCe$ and $C' = \CodeCe \oplus \Ang{u}$ for some even weight vector $u \in \F_2^n$.
	In other words, $C' = \CodeCe \cup (u + \CodeCe)$ for some coset $u + \CodeCe$ not contained in $C$.
	Because $C$ is LCD, for each $u \in \F_2^n$, we can write $u = c + d$ for $c \in C$ and $d \in C^{\perp}$.
	Furthermore, for some fixed $v \in C \setminus \CodeCe$, we can write $c \in C$ as $c = e + v$ for $e \in \CodeCe$.
	So the cosets of $\CodeCe$ correspond to the vectors in $C^{\perp} \cup (v + C^{\perp})$.

	We now show that only cosets corresponding to $v + C^{\perp}$ give us LCD codes.
	Assume that $C' = \CodeCe \oplus \Ang{d}$ for some nonzero (necessarily even-weight) $d \in C^{\perp}$.
	Then for all $c' \in \CodeCe$, $\Ang{c', d} = 0$.
	Furthermore, because it is of even weight, $\Ang{d, d} = 0$.
	So $d \in C' \cap C'^{\perp}$ and therefore $C'$ cannot be an LCD code.

	Finally, $v$ must be of odd weight in order for $C$ to be odd-like, $d$ must also be of odd weight in order for $u$ to be of even weight and $C'$ to be even-like.
\end{proof}

\begin{lemma}\label{lem:ooeo-suf}
	Let $k$ be even, $C \in \LCDoo{n}{k}$, and let $\CodeCe$ be its unique $(k - 1)$-dimensional even-like subcode. Furthermore, let $v \in C \setminus \CodeCe$ be a fixed vector such that $C = \CodeCe \oplus \Ang{v}$.
	Then for each odd $d \in C^{\perp}$, $C' = \CodeCe \oplus \Ang{v + d}$ is a unique even-like LCD code.
\end{lemma}

\begin{proof}
	Given some odd-weight $d \in C^{\perp}$, assume that $C' = \CodeCe \oplus \Ang{v + d}$ is not an LCD code.
	Then there exists some nonzero $c \in (C' \cap C'^{\perp})$.
	We consider two cases:
	\begin{itemize}
		\item Case 1: Assume $c \in \CodeCe$.
			Then
			\begin{align*}
				0
				= \Ang{c, v + d}
				= \Ang{c, v} + \Ang{c, d}
				= \Ang{c, v}.
			\end{align*}
			Since $c \in C'^{\perp} \subset \CodeCe^{\perp}$, we also have that $\Ang{c, c'} = 0$ for all $c' \in \CodeCe$.
			Therefore $c \in (C \cap C^{\perp})$ which is a contradiction since $C$ is an LCD code.
		\item Case 2: Assume $c = e + v + d$ for $e \in \CodeCe$.
			Again, since $c \in C'^{\perp} \subset \CodeCe^{\perp}$, we have that for all $c' \in \CodeCe$
			\begin{align*}
				0
				= \Ang{c, c'}
				= \Ang{e + v + d, c'}
				= \Ang{e + v, c'} + \Ang{d, c'}
				= \Ang{e + v, c'},
			\end{align*}
			i.e., $e + v \in \CodeCe^{\perp}$.

			The code $\CodeCe$ is even-like and has dimension $k - 1$ so it cannot be an LCD code.
			Therefore there exists some nonzero $w \in (\CodeCe \cap \CodeCe^{\perp})$.
			Then
			\begin{align*}
				0
				= \Ang{w, e + v}
				= \Ang{w, e} + \Ang{w, v}
				= \Ang{w, v}.
			\end{align*}
			and $\Ang{w, c'} = 0$ for all $c' \in \CodeCe$.
			So $w \in (C \cap C^{\perp})$ which is again a contradiction.
	\end{itemize}

	The uniqueness of the code $C'$ follows from the fact that the cosets of $\CodeCe$ are disjoint.
\end{proof}

\begin{proposition}\label{prop:ooeo-edges}
	Let $C \in \LCDoo{n}{k}$.
	Then $C$ has $2^{n - k - 1}$ neighbors in $\LCDeo{n}{k}$ if $k$ is even and none if $k$ is odd.
\end{proposition}

\begin{proof}
	Let $k$ be even and let $\CodeCe$ be the unique $(k - 1)$-dimensional even-like subcode of $C$ and fix some $v \in C \setminus \CodeCe$.
	By Lemma~\ref{lem:even-subcode}, since $C^{\perp}$ is odd-like, it contains $2^{n - k - 1}$ odd-weight codewords.
	Then by Lemma~\ref{lem:ooeo-nec} and Lemma~\ref{lem:ooeo-suf}, each odd-weight vector $d \in C^{\perp}$ corresponds to a coset $v + d + \CodeCe$ which in turn corresponds to a unique even-like neighboring LCD code.

	If $k$ is odd, there are no even LCD codes and therefore $C$ can have no neighbors which are simultaneously even-like and LCD\@.
\end{proof}

By considering the dual (see Proposition~\ref{prop:dual-neighbors}), the following corollary follows directly from Proposition~\ref{prop:ooeo-edges}.

\begin{corollary}\label{cor:oooe-edges}
	Let $C \in \LCDoo{n}{k}$.
	Then $C$ has $2^{k - 1}$ neighbors in $\LCDoe{n}{k}$ if $n - k$ is even and none if $n - k$ is odd.
\end{corollary}

\begin{proposition}\label{prop:oeoo-edges}
	Let $C \in \LCDoe{n}{k}$.
	Then $C$ has $2^{n - 1} - 2^{k - 1}$ neighbors in $\LCDoo{n}{k}$.
\end{proposition}

\begin{proof}
	Let $C \in \LCDoe{n}{k}$.
	Then, for $\LCDoe{n}{k}$ to be nonempty, $n - k$ must be even and necessarily $w = (1, \ldots, 1) \in C$.
	$C$ has $\Qbinom{k}{k - 1}{2} - \Qbinom{k - 1}{k - 2}{2} = 2^{k - 1}$ $(k - 1)$-dimensional subcodes not containing the codeword $w$.
	Any possible neighbor of $C$ in $\LCDoo{n}{k}$ can then be written as the union of such a subcode $C'$ with one of its cosets $v + C'$ for some $v \in \F_2^n$.
	Since $C$ is LCD, either $v$ is of the form $c + d$ or $c + d + w$ for some $c \in C'$ and some $d \in C^\perp$.
	Because $C^\perp \in \LCDeo{n}{k}$, $\Ang{d, d} = 0$, so $C' \oplus \Ang{d}$ would not be an LCD code.
	Furthermore, the code $C' \oplus \Ang{w}$ would have an even dual so would not be contained in $\LCDoo{n}{k}$.
	Therefore any neighboring LCD codes in $\LCDoo{n}{k}$ must be of the form $C' \oplus \Ang{d + w}$ and $C$ has at most $2^{k - 1} (2^{n - k} - 1)$ of them.

	We now show that each such $C$ must have exactly this many.
	Consider the bipartite graph whose independent sets are the sets $\LCDoo{n}{k}$ and $\LCDoe{n}{k}$ and which has an edge between vertices when a code in $\LCDoo{n}{k}$ is neighbors with a code in $\LCDoe{n}{k}$.
	By Corollary~\ref{cor:oooe-edges}, each vertex of $\LCDoo{n}{k}$ is $2^{k - 1}$-regular and therefore the graph has $2^{k - 1} \Abs{\LCDoo{n}{k}}$ edges.
	This means that the average degree of each vertex of $\LCDoe{n}{k}$ is
	\begin{align*}
		d_{n, k} = 2^{k - 1} \frac{\Abs{\LCDoo{n}{k}}}{\Abs{\LCDoe{n}{k}}}.
	\end{align*}
	Because $n - k$ must be even, we consider the cases where both $n$ and $k$ are odd and where both are even.
	We use Theorem~4.6 from~\cite{CMTQ18} (restated in this paper as Theorem~\ref{thm:CMTQ-count}).
	\begin{itemize}
		\item Case $n$ odd and $k$ odd:
			\begin{align*}
				d_{n, k}
				&= 2^{k - 1} \frac{\Abs{\LCDoo{n}{k}}}{\Abs{\LCDoe{n}{k}}} \\
				&= 2^{k - 1} \frac{%
					2^{\frac{(n - k)(k - 1)}{2}}
					(2^{n - k} - 1)
					\Qbinom{\tfrac{n - 1}{2}}{\tfrac{k - 1}{2}}{4}
				}{%
					2^{\frac{(n - k)(k - 1)}{2}}
					\Qbinom{\tfrac{n - 1}{2}}{\tfrac{k - 1}{2}}{4}
				} \\
				&= (2^{n - k} - 1) 2^{k - 1} \\
				&= 2^{n - 1} - 2^{k - 1}.
			\end{align*}
		\item Case $n$ even and $k$ even:
			\begin{align*}
				d_{n, k}
				&= 2^{k - 1} \frac{\Abs{\LCDoo{n}{k}}}{\Abs{\LCDoe{n}{k}}} \\
				&= 2^{k - 1} \frac{%
					2^{\frac{k (n - k)}{2}}
					(2^k - 1)
					\Qbinom{\tfrac{n}{2} - 1}{\tfrac{k}{2}}{4}
				}{%
					2^{\frac{k (n - k)}{2}}
					\Qbinom{\tfrac{n}{2} - 1}{\tfrac{k}{2} - 1}{4}
				} \\
				&= (2^k - 1) 2^{k - 1} \frac{%
					\Qbinom{\tfrac{n}{2} - 1}{\tfrac{k}{2}}{4}
				}{%
					\Qbinom{\tfrac{n}{2} - 1}{\tfrac{k}{2} - 1}{4}
				} \\
				&= (2^k - 1) 2^{k - 1} \frac{%
					4^{\tfrac{n}{2} - \tfrac{k}{2}} - 1
				}{%
					4^{\tfrac{k}{2}} - 1
				} \\
				&= 2^{k - 1} (2^{n - k} - 1) \\
				&= 2^{n - 1} - 2^{k - 1}.
			\end{align*}
	\end{itemize}
	So, in both cases the average degree of the vertices in $\LCDoe{n}{k}$ is equal to the upper bound proving the result.
\end{proof}

Again, by considering the dual (see Proposition~\ref{prop:dual-neighbors}), the following corollary follows directly from Proposition~\ref{prop:oeoo-edges}.

\begin{corollary}\label{cor:eooo-edges}
	Any $C \in \LCDeo{n}{k}$ has $2^{n - 1} - 2^{n - k - 1}$ neighbors in $\LCDoo{n}{k}$.
\end{corollary}

We have shown that any bipartite graph whose independence sets are $\LCDoo{n}{k}$, $\LCDoe{n}{k}$, or $\LCDeo{n}{k}$ and whose edges are neighbor relations between vertices of the sets is biregular.

In the following lemmata, we give a specific characterization of the neighbors of binary LCD codes which will then be used to calculate the various regularity degrees of the subgraphs corresponding to the three orbits.

\begin{lemma}\label{lem:augmented-binary}
	Let $C$ be an ${[n, k]}_2$ binary LCD code with generator matrix $G$ and let $v \in \F_2^n$ such that $v = v|_C + v|_{C^{\perp}}$ where $v|_C \in C$ and $v|_{C^{\perp}} \in C^{\perp}$.
	Then, the code $C'$ generated by the matrix
	\begin{align*}
		G' = \begin{bmatrix}
			G \\
			v
		\end{bmatrix}
	\end{align*}
	is an ${[n, k + 1]}_2$ binary LCD code if and only if $\Ang{v|_{C^{\perp}}, v|_{C^{\perp}}} \neq 0$.
\end{lemma}

\begin{proof}
	If $v = v|_C + v|_{C^{\perp}}$, then $G' = \begin{bmatrix}
			G \\
			v
		\end{bmatrix}$ and $G'' = \begin{bmatrix}
			G \\
			v|_{C^{\perp}}
		\end{bmatrix}$ generate the same code.
	Furthermore,
	$C'$ is LCD of dimension $k + 1$ if and only if $\det(G'' \cdot G''^\top) \neq 0$, by~\cite[Theorem 2.1]{Mas92}.
	Note that
	\begin{align*}
		\det(G'' \cdot G''^\top)
		&= \det\left( \begin{bmatrix}
			G \cdot G^\top & G \cdot v|_{C^{\perp}}^\top \\
			v|_{C^{\perp}} \cdot G^\top & v|_{C^{\perp}} \cdot v|_{C^{\perp}}^\top
		\end{bmatrix}\right) \\
		&= \det\left( \begin{bmatrix}
			G \cdot G^\top & 0 \\
			 0 &  v|_{C^{\perp}} \cdot v|_{C^{\perp}}^\top
		\end{bmatrix}\right) \\
		&= \det(G \cdot G^\top) \cdot \Ang{v|_{C^{\perp}}, v|_{C^{\perp}}}.
	\end{align*}
	Since $C$ is LCD we have $\det(G \cdot G^\top) \neq 0$ and thus $\det(G'' \cdot G''^\top) \neq 0$ if and only if $\Ang{v|_{C^{\perp}}, v|_{C^{\perp}}} \neq 0$.
\end{proof}

\begin{lemma}\label{lem:half-neighbors}
	Let $C$ be a binary $[n, k]$ LCD code and let $C' \subset C$ be any $(k - 1)$-dimensional subcode with $C' \oplus \Ang{v} = C$ for some vector $v \in C \setminus C'$.
	Then for all nonzero $d \in C^\perp$, exactly one of $C' \oplus \Ang{d}$ and $C' \oplus \Ang{d + v}$ is an LCD code.
\end{lemma}

\begin{proof}
	Let $G'$ be a generator matrix for the subcode $C'$, let $C_d = C' \oplus \Ang{d}$, and let $C_{d + v} = C' \oplus \Ang{d + v}$.
	First  assume $C_d$ is an LCD code.
	Then $C_d$ is generated by the matrix
	\begin{align*}
		G_d = \begin{bmatrix}
			G' \\
			d
		\end{bmatrix}
	\end{align*}
	and we have the result that $G_d \cdot G_d^\top$ is invertible; i.e.,
	\begin{align*}
		1
		= \det(G_d \cdot G_d^\top)
		= \det \left( \begin{bmatrix}
			G' \cdot G'^\top & G' \cdot d^\top \\
			d \cdot G'^\top & d \cdot d^\top
		\end{bmatrix} \right)
		= \det \left( \begin{bmatrix}
			G' \cdot G'^\top & 0 \\
			0 & d \cdot d^\top
		\end{bmatrix} \right).
	\end{align*}
	It then follows that $\det(G' \cdot G'^\top) = 1$, i.e., $C'$ is an LCD code.
	Furthermore $d \cdot d^\top = \Ang{d, d} = 1$.
	Since $C'$ is LCD, let $v' = v |_{C'^\perp}$ and note that $C' \oplus \Ang{v'} = C' \oplus \Ang{v} = C$.
	Then, because $C = C' \oplus \Ang{v'}$ is an LCD code and $v' \in C'^\perp$, by Lemma~\ref{lem:augmented-binary}, we know that $\Ang{v', v'} = 1$ and, furthermore, because $\Ang{d + v', d + v'} = \Ang{d, d} + \Ang{v', v'} = 0$ and because both $d, v' \in C'$, it follows that
	\begin{align*}
		\det \left( \begin{bmatrix}
			G' \\
			d + v'
		\end{bmatrix} \cdot {\begin{bmatrix}
			G' \\
			d + v'
		\end{bmatrix}}^\top \right)
		&= \det \left( \begin{bmatrix}
			G' \cdot G'^\top & G' \cdot {(d + v')}^\top \\
			(d + v') \cdot G'^\top & (d + v') \cdot {(d + v')}^\top
		\end{bmatrix} \right) \\
		&= \det \left( \begin{bmatrix}
			G' \cdot G'^\top & 0 \\
			0 & 0
		\end{bmatrix} \right)
		= 0.
	\end{align*}
	Hence, $C_{d + v} = C' \oplus \Ang{d + v} = C' \oplus \Ang{d + v'}$ cannot be an LCD code.

	Now assume $C_d$ is not an LCD code.
	Then
	\begin{align*}
		\det \left( \begin{bmatrix}
			G' \cdot G'^\top & 0 \\
			0 & d \cdot d^\top
		\end{bmatrix} \right)
		= 0
	\end{align*}
	and so either $C'$ is not an LCD code or $\Ang{d, d} = 0$ (or both).
	If $\Ang{d, d} = 0$, then $\Ang{d + v, d + v} = \Ang{v, v}$ and
	\begin{align*}
		\det \left( \begin{bmatrix}
			G' \\
			d + v
		\end{bmatrix} \cdot {\begin{bmatrix}
			G' \\
			d + v
		\end{bmatrix}}^\top \right)
		&= \det \left( \begin{bmatrix}
			G' \cdot G'^\top & G' \cdot {(d + v)}^\top \\
			(d + v) \cdot G'^\top & (d + v) \cdot {(d + v)}^\top
		\end{bmatrix} \right) \\
		&= \det \left( \begin{bmatrix}
			G' \cdot G'^\top & G' \cdot v^\top \\
			v \cdot G'^\top & v \cdot v^\top
		\end{bmatrix} \right) \\
		&= \det \left( \begin{bmatrix}
			G' \\
			v
		\end{bmatrix} \cdot {\begin{bmatrix}
			G' \\
			v
		\end{bmatrix}}^\top \right)
		= 1.
	\end{align*}
	Therefore $C_{d + v}$ is also an LCD code.

	So, now assume that $\Ang{d, d} = 1$ and that $C'$ is not an LCD code.
	Then $\det(G' \cdot G'^\top) = 0$ and
	\begin{align*}
		\det \left( \begin{bmatrix}
			G' \\
			d + v
		\end{bmatrix} \cdot {\begin{bmatrix}
			G' \\
			d + v
		\end{bmatrix}}^\top \right)
		&= \det \left( \begin{bmatrix}
			G' \cdot G'^\top & G' \cdot {(d + v)}^\top \\
			(d + v) \cdot G^\top & (d + v) \cdot {(d + v)}^\top
		\end{bmatrix} \right) \\
		&= \det \left( \begin{bmatrix}
			G' \cdot G'^\top & G' \cdot v^\top \\
			v \cdot G^\top & 1 + v \cdot v^\top
		\end{bmatrix} \right) \\
		&= \det \left( \begin{bmatrix}
			G' \\
			v
		\end{bmatrix} \cdot {\begin{bmatrix}
			G' \\
			v
		\end{bmatrix}}^\top \right) + 1 \cdot \det \left( G' \cdot G'^\top \right)
		= 1 + 1 \cdot 0
		= 1.
	\end{align*}
	In other words, the matrices $\begin{bmatrix} G' \\ d + v \end{bmatrix} \cdot {\begin{bmatrix} G' \\ d + v \end{bmatrix}}^\top$ and $\begin{bmatrix} G' \\ v \end{bmatrix} \cdot {\begin{bmatrix} G' \\ v \end{bmatrix}}^\top$ differ only in their bottom-right-most entry and the value of this entry does not affect the matrices' determinant since $\det(G' \cdot G'^\top) = 0$.
	So $C_{d + v}$ is an LCD code.
\end{proof}

Using the previous lemma, we can recover the number of LCD neighbors (for the binary case) from Theorem~\ref{thm:LCD-neighbors} which we restate here.

\begin{theorem}
	An ${[n, k]}_2$ LCD code $C$ has exactly $(2^{n - k} - 1)(2^k - 1)$ neighbors which are also LCD codes.
\end{theorem}

\begin{proof}
	This follows directly from Proposition~\ref{prop:all-neighbors} and Lemma~\ref{lem:half-neighbors}.
\end{proof}

In addition to the biregularity degrees given by Propositions~\ref{prop:ooeo-edges} and~\ref{prop:oeoo-edges}, and Corollaries~\ref{cor:oooe-edges} and~\ref{cor:eooo-edges} we finally calculate the regularities within the subgraphs.

\begin{theorem}\label{thm:q2Subgraphsregularity}
	Let $\Goo{n}{k}$, $\Goe{n}{k}$, and $\Geo{n}{k}$ be (respectively) the subgraphs of the LCD neighbor graph consisting of the codes in $\LCDoo{n}{k}$, $\LCDoe{n}{k}$, and $\LCDeo{n}{k}$.
	Then all three are regular with degree
	\begin{align*}
		\Gd{\Goo{n}{k}} &= \begin{cases}
			(2^{n - k} - 1)(2^k - 1) - 2^{n - k - 1} - 2^{k - 1}
			&\text{ if $n$ is even and $k$ is even}, \\
			(2^{n - k} - 1)(2^k - 1)
			&\text{ if $n$ is even and $k$ is odd}, \\
			(2^{n - k} - 1)(2^k - 1) - 2^{n - k - 1}
			&\text{ if $n$ is odd and $k$ is even}, \\
			(2^{n - k} - 1)(2^k - 1) - 2^{k - 1}
			&\text{ if $n$ is odd and $k$ is odd};
		\end{cases} \\
		\Gd{\Geo{n}{k}} &= \begin{cases}
			(2^{n - k} - 1)(2^k - 1) - 2^{n - 1} + 2^{n - k - 1}
			&\text{ if $k$ is even}, \\
			0
			&\text{ if $k$ is odd};
		\end{cases} \\
		\Gd{\Goe{n}{k}} &= \begin{cases}
			(2^{n - k} - 1)(2^k - 1) - 2^{n - 1} + 2^{k - 1}
			&\text{ if $n - k$ is even}, \\
			0
			&\text{ if $n - k$ is odd}.
		\end{cases}
	\end{align*}
\end{theorem}

\begin{proof}
	This follows directly from the regularity of the full graph and the biregularity---see Theorem~\ref{thm:biregular}---when only considering the edges between the subgraphs.
	In the case the subgraph is empty, we say the graph has regularity zero.
\end{proof}


\section{Structure of LCD neighbor graphs for odd \texorpdfstring{$\mathbf{q}$}{q}}\label{sec:Structureqodd}

First, analogously to the binary case, we show that the LCD neighbor graph for odd $q$ is regular and connected.
The regularity follows again straightforwardly from Theorem~\ref{thm:LCD-neighbors}:

\begin{corollary}\label{cor:regular_odd}
	For odd $q$, the LCD neighbor graph $\Gq_q(n, k)$ is regular of degree 
	\begin{align*}
		r := \frac{(q^k - 1)(q^{n - k + 1} - 1)}{q (q - 1)} - \begin{cases}
			\Leg{-1}{q}^\frac{n}{2} q^{\frac{n}{2} - 1} & \text{ if $n$ is even and $k$ is odd}\\
			0 & \text{ otherwise}
		\end{cases}.
	\end{align*}
\end{corollary}

\begin{theorem}
	For odd $q$, the LCD neighbor graph $\Gq_q(n, k)$ is connected.
\end{theorem}

\begin{proof}
	As in the binary case, we show the regularity degree $r$ is strictly greater than $\lambda_2(\Jq_q(n, k))$, which implies the connectedness by Proposition~\ref{prop:spec-connect} and Theorem~\ref{thm:interlacing}. 
	 The regularity from Corollary~\ref{cor:regular_odd} is lower bounded by 
	\begin{align*}
		r':= \frac{(q^k - 1)(q^{n - k + 1} - 1)}{q (q - 1)} - q^{\frac{n}{2} - 1},
	\end{align*}
	and we will show $q {(q - 1)}^2 (r' - \lambda_2(\Jq_q(n, k))) > 0$ (the factor $q {(q - 1)}^2 $ is for computational ease). 
	By Theorem~\ref{thm:Jq-char},
	\begin{align*}
		& q  {(q - 1)}^2 (r' - \lambda_2(\Jq_q(n, k))) \\
		=& (q - 1) (q^k - 1)(q^{n - k + 1} - 1) - {(q - 1)}^2 q^{\frac{n}{2}}  - q^3 (q^{k - 1} - 1) (q^{n - k - 1} - 1) - q {(q - 1)}^2 \\
		=& \underbrace{q^{\frac{n}{2}} (q(q - 2) q^{\frac{n}{2}} - {(q - 1)}^2)}_{\geq 15}
		+ \underbrace{q^{n - k + 1}}_{\geq 9}
		+ \underbrace{(q^k - 2)(q^2 - q + 1)}_{\geq 7} + 1
		> 0.
	\end{align*}
\end{proof}

In the following, we show that, for odd $q$, any LCD code has at least one neighbor in the other orbit.

\begin{lemma}\label{lem:orbit-neighbors}
	Let $q$ be odd and $\Orb_{\Ort_n(\F_q)}(C_+) \cup \Orb_{\Ort_n(\F_q)}(C_-)$ the orbit decomposition of all ${[n, k]}_q$ LCD codes, for some $k < n$.
	Then any element of $\Orb_{\Ort_n(\F_q)}(C_+)$ has a neighbor on $\Orb_{\Ort_n(\F_q)}(C_-)$.
\end{lemma}

\begin{proof}
	We know from~\cite{CMTQ18} that $\Orb_{\Ort_n(\F_q)}(C_+)$ is generated by the code $C_+$ with generator matrix $G_+ := (I_k \mid 0)$ and $\Orb_{\Ort_n(\F_q)}(C_-)$ is generated by the code $C_-$ with generator matrix 
	\begin{align*}
		G_- := \begin{bmatrix}
			I_{k-1} &   & 0 &   \\
			0	   & a & b & 0
		\end{bmatrix}
	\end{align*}
	for some $a,b\in \F_q$ with $a^2 + b^2$ being a non-square\footnote{These always exist, see~\cite{CMTQ18}.}.
	One can easily see that $C_+$ and $C_-$ are neighbors and we have---as in the binary case---$\dim(C_+ \cdot M \cap C_- \cdot M)=k-1$ for any $M\in \Ort_n(\F_q)$.
\end{proof}

Note that the proof above also shows how to compute the neighbor of any element $C_+ M \in \Orb_{\Ort_n(\F_q)}(C_+)$ on the other orbit as $C_- M \in \Orb_{\Ort_n(\F_q)}(C_-)$.

In the remainder of the section, we will calculate the regularity and biregularity degrees for the odd $q$ subgraphs.
Theorem~\ref{thm:LCD-equivalence} tells us that $C$ is LCD if and only if $\Sign(C) = \pm 1$.
We follow the notation of~\cite{CMTQ18} and define $\LCDpq{n}{k}$ and $\LCDnq{n}{k}$ to be the sets of ${[n, k]}_q$ LCD codes, respectively, with sign $+1$ and $-1$. 
Additionally, for $z \in \N$, let
\begin{align*}
	\Tauq{z} := \Leg{{(-1)}^{\left\lceil \frac{z - 1}{2} \right\rceil}}{q}.
\end{align*}

The following corollary follows directly from Lemma~\ref{lem:solutions-quadratic-form}.

\begin{corollary}\label{cor:signed-vectors}
	Let $q$ be odd and $C \in \LCDq{n}{k}$.
	Let $\CurV_0(C)$, $\CurV_+(C)$, and $\CurV_-(C)$ respectively represent the sets of self-orthogonal vectors in $C$, and the number of vectors $v$ in $C$ with $\Leg{\Ang{v, v}}{q} = \pm 1$.
	Then the number of self-orthogonal vectors in $C$ is
	\begin{align*}
		\Abs{\CurV_0(C)} = 
		\begin{cases}
			q^{k - 1} + \Sign(C) \cdot \Tauq{k} \cdot (q - 1) q^{\frac{k - 2}{2}}
			&\text{if $k$ is even,} \\
			q^{k - 1}
			&\text{if $k$ is odd,}
		\end{cases}
	\end{align*}
	and the number of vectors $v \in C$ with $\Leg{\Ang{v, v}}{q} = t \neq 0$ is
	\begin{align*}
		\Abs{\CurV_t(C)} = 
		\begin{cases}
			\frac{q - 1}{2} \cdot
			\left( q^{k - 1} - \Sign(C) \cdot \Tauq{k} \cdot q^{\frac{k - 2}{2}} \right)
			&\text{if $k$ even,} \\
			\frac{q - 1}{2} \cdot
			\left( q^{k - 1} + t \cdot \Sign(C) \cdot \Tauq{k} \cdot q^{\frac{k - 1}{2}} \Tauq{k} \right)
			&\text{if $k$ odd.}
		\end{cases}
	\end{align*}
\end{corollary}

In order to count the number of subcodes of an LCD code with a given sign, we need the following two lemmata.

\begin{lemma}\label{lem:subcode-construct}
	Let $q$ be odd, $C \in \LCDq{n}{k}$, and $S \subsetneq C$ be a $(k - 1)$-dimensional subcode of $C$.
	Then $S = {(C^\bot \oplus \Ang{v})}^\bot$ for some $v \in C$.
	Furthermore, distinct subspaces $\Ang{v}$ determine unique subcodes.
\end{lemma}

\begin{proof}
	Let $C$ have parity-check matrix $H$.
	Then $S$ can be determined by parity-check matrix $\begin{bmatrix} H \\ v \end{bmatrix}$ for some $v \in \F_q^n \setminus C^\bot$.
	Since $C$ is LCD, we can write $v = v_0 + v_\bot$ for $v_0 \in C$ and $v_\bot \in C^\bot$.
	Then $\begin{bmatrix} H \\ v_0 \end{bmatrix}$ is also a parity-check matrix for $S$ and therefore $S = {(C^\bot \oplus \Ang{v_0})}^\bot$.
	
		Finally, assume, there exist $u, v \in C$ such that $\Ang{u} \neq \Ang{v}$ but that $C^\bot \oplus \Ang{u} = C^\bot \oplus \Ang{v}$.
		Then $u = w + \alpha v$ for some nonzero $w \in C^\bot$ and some $\alpha \in \F_q$.
		But then $w \in (C \cap C^\bot)$ which is a contradiction since $C$ is an LCD code.
\end{proof}

\begin{lemma}\label{lem:augment-sign}
	Let $q$ be odd, $C \in \LCDq{n}{k}$, and $v \in C^\bot$.
	Then $\Sign(C \oplus \Ang{v}) = \Sign(C) \cdot \Leg{\Ang{v, v}}{q}$.
\end{lemma}

\begin{proof}
	We consider two cases.
	\begin{itemize}
		\item Case $\Ang{v, v} = 0$: then, since $v \in C \oplus \Ang{v}$ and $v \in C^\bot$, $v \in (C \oplus \Ang{v}) \cap {(C \oplus \Ang{v})}^\bot$ so $\Sign(C \oplus \Ang{v}) = 0$.
		\item Case $\Leg{\Ang{v, v}}{q} = \pm 1$: let $G$ be a generator matrix for $C$.
			Then $vG^\top=0$ and
			\begin{align*}
				\Sign(C \oplus \Ang{v})
				&= \Leg{\det \left( \begin{bmatrix} G \\ v \end{bmatrix} {\begin{bmatrix} G \\ v \end{bmatrix}}^\top \right)}{q}
				= \Leg{\det \left( \begin{bmatrix} G G^\top & 0 \\ 0 & \Ang{v, v} \end{bmatrix} \right)}{q} \\
				&= \Leg{\det \left( G G^\top \right)}{q} \cdot \Leg{\Ang{v, v}}{q}
				= \Sign(C) \cdot \Leg{\Ang{v, v}}{q}.
				\qedhere{}
			\end{align*}
	\end{itemize}
\end{proof}

\begin{theorem}\label{thm:subcode-count}
	Let $q$ be odd and $C \in \LCDq{n}{k}$.
	For $t \in \{-1, 0, 1\}$, define $\CurS_t(C) = \{S \subsetneq C | \dim(S) = k - 1, \Sign(S) = t\}$ to be the set of $(k - 1)$-dimensional subcodes of $C$ of a given sign.
	Then
	\begin{align*}
		\Abs{\CurS_0(C)} &= \begin{cases}
			\frac{q^{k - 1} - 1}{q - 1}
			+ \Sign(C) \cdot \Tauq{k} \cdot q^{\frac{k - 2}{2}}
			&\text{for $k$ even,} \\
			\frac{q^{k - 1} - 1}{q - 1}
			&\text{for $k$ odd;}
		\end{cases} \\
		\Abs{\CurS_+(C)} &= \begin{cases}
			\tfrac{1}{2} q^{k - 1}
			- \tfrac{1}{2} \cdot \Sign(C) \cdot \Tauq{k} \cdot q^{\frac{k - 2}{2}}
			&\text{for $k$ even,} \\
			\tfrac{1}{2} q^{k - 1}
			+ \tfrac{1}{2} \cdot \Tauq{k} \cdot q^{\frac{k - 1}{2}}
			&\text{for $k$ odd;}
		\end{cases} \\
		\Abs{\CurS_-(C)} &= \begin{cases}
			\tfrac{1}{2} q^{k - 1}
			- \tfrac{1}{2} \cdot \Sign(C) \cdot \Tauq{k} \cdot q^{\frac{k - 2}{2}}
			&\text{for $k$ even,} \\
			\tfrac{1}{2} q^{k - 1}
			- \tfrac{1}{2} \cdot \Tauq{k} \cdot q^{\frac{k - 1}{2}}
			&\text{for $k$ odd.}
		\end{cases}
	\end{align*}
\end{theorem}

\begin{proof}
	We construct the sets of subcodes by choosing vectors as in Lemma~\ref{lem:subcode-construct}.
	We then determine the sign of the subcode from the vector with which we augment the dual using Proposition~\ref{prop:dual-sign} and Lemma~\ref{lem:augment-sign}.
	Then, for calculating the number of subcodes of given sign $t$, we count the number of 1-dimensional subspaces generated by counting the vectors which generate them as given by Corollary~\ref{cor:signed-vectors}, i.e., by taking the number of vectors, removing the zero vector in the self-orthogonal case, and dividing by $(q - 1)$.
\end{proof}

We now consider the ways in which we can augment these subcodes in order to build neighboring LCD codes.

\begin{lemma}\label{lem:non-lcd-subcode-augmentation}
	Let $q$ be odd and $C$ be an ${[n, k]}_q$ code with $C \cap C^\bot = \Ang{u}$ for some vector $u$.
	Then $C \oplus \Ang{v}$ is an LCD code if and only if $\Ang{u, v} \neq 0$.
	Furthermore, if it is an LCD code, then $\Sign(C \oplus \Ang{v})$ is independent of the choice of $v$.
\end{lemma}

\begin{proof}
	First note that it follows necessarily that $\Ang{u, u} = 0$ and that for any $v \in \F_q^n$, $\Ang{u, v} \neq 0$ implies that $v \notin C \cup C^\bot$.
	Consider the generator matrices for $C$ and $C \oplus \Ang{v}$ of the following form:
	\begin{align*}
		G = \begin{bmatrix} G' \\ u \end{bmatrix},
		\qquad
		\hat{G} = \begin{bmatrix} G' \\ u \\ v \end{bmatrix}.
	\end{align*}
	Then
	\begin{align*}
		\det(\hat{G} \hat{G}^\top)
		&= \det \left( \begin{bmatrix}
			G' G'^\top & u G'^\top & v G'^\top \\
			G' u^\top  & u u^\top  & v u^\top  \\
			G' v^\top  & u v^\top  & v v^\top
		\end{bmatrix} \right)
		= \det \left( \begin{bmatrix}
			G' G'^\top & 0		  & v G'^\top  \\
			0		  & 0		  & \Ang{u, v} \\
			G' v^\top  & \Ang{u, v} & \Ang{v, v}
		\end{bmatrix} \right) \\
		&= {\Ang{u, v}}^2 \det(G' G'^\top).
	\end{align*}
	Consider then the code $C'$ generated by $G'$ and assume it is not an LCD code.
	Then there exists some $u' \in C' \subsetneq C$ such that $u' G'^\top = 0$.
	Since $u \in C^\bot$, we have $\Ang{u, u'} = 0$ and so $u' \in C^\bot$ as well.
	This is a contradiction since $\dim(C \cap C^\bot) = 1$ and $u' \notin \Ang{u}$.
	So $C'$ must be an LCD code and $\det(G' G'^\bot) \neq 0$.
	Thus $C \oplus \Ang{v}$ must also be an LCD code if and only if $\Ang{u, v} \neq 0$.

	The question remains if the choice of basis $G'$ changes the sign of $C'$ and therefore also $C \oplus \Ang{v}$.
	Given two such bases $G_1'$ and $G_2'$, we know that for some invertible $(k - 1) \times (k - 1)$ matrix $M$ and a matrix $U = \begin{bmatrix} \alpha_1 u \\ \vdots \\ \alpha_{k - 1} u \end{bmatrix}$ for scalars $\alpha_1, \ldots, \alpha_{k - 1} \in \F_q$, we have $G_2' = M G_1' + U$.
	Then
	\begin{align*}
		\det(G_2' G_2'^\top)
		&= \det((M G_1' + U){(M G_1' + U)}^\top) \\
		&= \det(M G_1' G_1'^\top M^\top + U G_1'^\top M^\top + M G_1' U^\top + U U^\top) \\
		&= \det(M G_1' G_1'^\top M^\top)
		 = {\det(M)}^2 \det(G_1' G_1'^\top).
	\end{align*}
	Therefore
	\begin{align*}
		\Leg{\det(G_2' G_2'^\top)}{q}
		= \Leg{{\det(M)}^2}{q} \Leg{\det(G_1' G_1'^\top)}{q}
		= \Leg{\det(G_1' G_1'^\top)}{q}.
	\end{align*}
\end{proof}

We finally need one more lemma showing that the dimension of the intersection of an LCD code's subcode with its dual is at most 1.

\begin{lemma}\label{lem:subcode-hulldim}
	Let $q$ be odd, $C \in \LCDq{n}{k}$, and $v \in C$ be such that $\Ang{v, v} = 0$.
	Define $S = {(C^\bot \oplus \Ang{v})}^\bot$.
	Then $\dim(S \cap S^\bot) = 1$.
\end{lemma}

\begin{proof}
	For all $u \in C^\bot \oplus \Ang{v}$, $\Ang{u, v} = 0$ so $v \in S$.
	Assume there exists some $w \in S \setminus \Ang{v}$ which is also contained in $S^\bot$.
	Then $w \in C$ (since $S \subsetneq C$) and $w = w' + \alpha v$ for some $w' \in C^\bot$ and some $\alpha \in \F_q$.
	$v, w \in C$ implies that $w' = w - \alpha v \in (C \cap C^\bot)$ which is a contradiction.
\end{proof}

We can now state the main theorem of this section, determining the regularity degrees of the LCD neighbor graph for odd $q$.

\begin{theorem}\label{thm:odd-q-neighbor-counts}
	Let $C \in \LCDq{n}{k}$ with $r = \Sign(C)$ and let
	\begin{align*}
		\CurE_0 &= \frac{(q^{n - k} - 1)(q^k - 1)}{q - 1}; \\
		\CurE_r &= \begin{cases}
			\CurE_0 - q^{n - 1} + r \cdot \Tauq{k} \cdot q^{n - \tfrac{k}{2} - 1}
			&\text{if $k$ is even,} \\
			\CurE_0 - q^{n - 1}
			&\text{if $k$ is odd;}
		\end{cases} \\
		\CurN_{r, s} &= \begin{cases}
			\tfrac{1}{2} \cdot q^{\frac{n - 2}{2}} \left(
				\left( q^{\frac{k}{2}} - r \cdot \Tauq{k} \right)
				\left( q^{\frac{n - k}{2}} + s \cdot \Tauq{n - k + 1} \right)
			\right)
			&\text{for $n$ even, $k$ even,} \\
			\tfrac{1}{2} \cdot q^{\frac{n - 2}{2}} \left(
				q^{\frac{n}{2}} - \Tauq{n}
			\right)
			&\text{for $n$ even, $k$ odd,} \\
			\tfrac{1}{2} \cdot q^{n - \tfrac{k}{2} - 1} \left(
				q^{\frac{k}{2}} - r \cdot \Tauq{k}
			\right)
			&\text{for $n$ odd, $k$ even,} \\
			\tfrac{1}{2} \cdot q^{\frac{n + k - 2}{2}} \left(
				q^{\frac{n - k}{2}} + s \cdot \Tauq{n - k + 1}
			\right)
			&\text{for $n$ odd, $k$ odd.}
		\end{cases}
	\end{align*}
	Then $C$ has $\CurN_{r, r} + \CurE_r$ neighbors in $\mathrm{LCD}_r {[n, k]}_q$ and $\CurN_{r, -r}$ neighbors in $\mathrm{LCD}_{-r} {[n, k]}_q$, where $\mathrm{LCD}_t{[n, k]}_q$ is the subset of codes of $\mathrm{LCD}{[n, k]}_q$ whose sign is $t$.
\end{theorem}

\begin{proof}
	Given $C$, by Theorem~\ref{thm:subcode-count}, we know both the number of unique LCD subcodes with a given sign and the number of non-LCD subcodes which by Lemma~\ref{lem:subcode-hulldim} must have a 1-dimensional intersection with its dual.
	We then count the number of appropriate vectors with which we can augment these subcodes to give LCD codes (which aren't $C$).
	In the case of the non-LCD subcodes, we know by Lemma~\ref{lem:non-lcd-subcode-augmentation} that augmentation by any appropriate choice of vector gives the same sign and since this includes the code $C$, any such derived LCD code must have sign $r$.
	We then set the values $\CurN_{r, +}$ and $\CurN_{r, -}$ to respectively be the number of LCD codes with sign $+$ and $-$ obtained from augmenting LCD subcodes of dimension $k - 1$ (note that these values count $C$ as well); i.e., 
	\begin{align*}
		\CurN_{r, +}
		= \frac{1}{q - 1} \left(
			\Abs{\CurS_+(C) \vphantom{\CurV_+(C_+^\bot)}} \Abs{\CurV_+(C_+^\bot)} +
			\Abs{\CurS_-(C) \vphantom{\CurV_-(C_-^\bot)}} \Abs{\CurV_-(C_-^\bot)}
		\right), \\
		\CurN_{r, -}
		= \frac{1}{q - 1} \left(
			\Abs{\CurS_+(C) \vphantom{\CurV_-(C_+^\bot)}} \Abs{\CurV_-(C_+^\bot)} +
			\Abs{\CurS_-(C) \vphantom{\CurV_+(C_-^\bot)}} \Abs{\CurV_+(C_-^\bot)}
		\right),
	\end{align*}
	where $C_+$ and $C_-$ are arbitrary $(k - 1)$-dimensional subcodes of $C$ with sign $+$ and $-$.

	Then, looking at the case where the subcode---call it $C_0$---is not LCD, we take a generator $\Ang{u} = C_0 \cap C_0^\bot$ and count all vectors $v$ not orthogonal to $u$ which augment $C_0$ to give a unique code (i.e., we only take one coset representative each of $C_0$).
	So we take all vectors in $\F_q^n$, subtract out those which are orthogonal to $u$, divide by $\Abs{C_0} = q^{k - 1}$, and divide by $(q - 1)$ since the same 1-dimensional subspace will yield the same code.
	This gives us
	\begin{align*}
		\frac{q^n - q^{n - 1}}{(q - 1) q^{k - 1}} = q^{n - k} - 1.
	\end{align*}
	Finally, we define $\CurE_0$ as above (which happens to be the number of LCD neighbors from the simpler case in Theorem~\ref{thm:LCD-neighbors}) and define
	\begin{align*}
		\CurE_r = (q^{n - k} - 1) \Abs{\CurS_0(C)} - \Abs{\CurS_+(C)} - \Abs{\CurS_-(C)},
	\end{align*}
	where the first summand counts the number of LCD neighbors of sign $r$ coming from non-LCD subcodes and the other two summands subtract off one for each LCD subcode so that we do not count the cases where the subcode is augmented to reobtain $C$.
	Thus, counting neighbors with the same sign, the number is $\CurN_{r, r} + \CurE_r$ while for neighbors of the opposite sign, it is $\CurN_{r, -r}$.
\end{proof}

\begin{corollary}\label{cor:qOddproperties}
	With the notation given in Theorem~\ref{thm:odd-q-neighbor-counts}, for odd $q$, the ${[n, k]}_q$ LCD neighbor graph:
	\begin{itemize}
		\item is regular with degree given by Theorem~\ref{thm:LCD-neighbors};
		\item has regular subgraphs corresponding to vertices $\LCDpq{n}{k}$ and $\LCDnq{n}{k}$ with regularity degrees $\CurN_{+, +} + \CurE_+$ and $\CurN_{-, -} + \CurE_-$;
		\item when restricted to edges between $\LCDpq{n}{k}$ and $\LCDnq{n}{k}$, the graph is biregular with (respective) degrees $\CurN_{+, -}$ and $\CurN_{-, +}$.
	\end{itemize}
\end{corollary}


\section{Conclusion}\label{sec:Conclusions}

In order to better understand the structure of LCD codes, we analyzed their neighbor graphs (i.e., the induced subgraphs of the Grassmann graphs given by LCD codes).
We have shown that these graphs are themselves regular with degree given by Theorem~\ref{thm:LCD-neighbors} and connected.
We further analyzed these graphs by viewing the induced subgraphs given by the orbits of the action of the orthogonal group upon the codes, showing that these, in turn, are regular, that the bipartite graphs formed by the edges between orbits are biregular, and explicitly calculated these (bi)regularity degrees.


\printbibliography[heading=bibnumbered]{}


\end{document}

%% file: packages.tex
\usepackage[a4paper]{geometry} 

\usepackage[
	backend=biber,
	maxbibnames=99
]{biblatex}
\usepackage{amsfonts, amsgen, amsmath, amssymb, amsthm}
\usepackage{authblk}
\usepackage{booktabs}
\usepackage{calc}
\usepackage{enumerate}
\usepackage{hyperref}
\usepackage{latexsym}
\usepackage{mathtools}
\usepackage{xcolor}

\theoremstyle{definition}
\newtheorem{theorem}{Theorem}[section]

\newtheorem{corollary}[theorem]{Corollary}
\newtheorem{definition}[theorem]{Definition}
\newtheorem{example}[theorem]{Example}
\newtheorem{lemma}[theorem]{Lemma}

\newtheorem{proposition}[theorem]{Proposition}

\newtheorem{remark}[theorem]{Remark}

%% file: macros.tex
\newcommand{\F}{\mathbb{F}}
\newcommand{\N}{\mathbb{N}}

\newcommand{\Proj}{\mathbb{P}}

\newcommand{\Gd}[1]{\mathrm{d} \left(#1\right)}                        
\newcommand{\Leg}[2]{\left(\frac{#1}{#2}\right)}                       
\newcommand{\Orb}{\mathrm{Orb}}
\newcommand{\Qbinom}[3]{{\begin{bmatrix} #1 \\ #2 \end{bmatrix}}_{#3}} 

\newcommand{\Sign}{\mathrm{sign}}
\newcommand{\Spec}{\mathrm{Spec}}
\newcommand{\Wt}{\mathrm{wt}}

\newcommand{\Abs}[1]{\left|#1\right|}
\newcommand{\Ang}[1]{\left\langle#1\right\rangle}

\newcommand{\LCDoo}[2]{\mathrm{LCD}_{\mathrm{oo}} [#1, #2]}
\newcommand{\LCDoe}[2]{\mathrm{LCD}_{\mathrm{oe}} [#1, #2]}
\newcommand{\LCDeo}[2]{\mathrm{LCD}_{\mathrm{eo}} [#1, #2]}
\newcommand{\LCDee}[2]{\mathrm{LCD}_{\mathrm{ee}} [#1, #2]}

\newcommand{\LCDpq}[2]{\mathrm{LCD}_+ {[#1, #2]}_q}
\newcommand{\LCDnq}[2]{\mathrm{LCD}_- {[#1, #2]}_q}
\newcommand{\LCDq}[2]{\mathrm{LCD} {[#1, #2]}_q}

\newcommand{\CodeCe}{C_{\mathrm{e}}}
\newcommand{\CodeCeo}{C_{\mathrm{eo}}}
\newcommand{\CodeCoo}{C_{\mathrm{oo}}}
\newcommand{\CodeCoe}{C_{\mathrm{oe}}}

\newcommand{\MatGoo}{G_{\mathrm{oo}}}
\newcommand{\MatGoe}{G_{\mathrm{oe}}}

\newcommand{\CurB}{\mathcal{B}}

\newcommand{\CurE}{\mathcal{E}}

\newcommand{\CurJ}{\mathcal{J}}

\newcommand{\CurN}{\mathcal{N}}
\newcommand{\CurS}{\mathcal{S}}
\newcommand{\CurV}{\mathcal{V}}

\newcommand{\Jq}{\CurJ}

\newcommand{\Gq}{\mathfrak{G}}
\newcommand{\Goo}[2]{\Gq_{\mathrm{oo}} (#1, #2)}
\newcommand{\Goe}[2]{\Gq_{\mathrm{oe}} (#1, #2)}
\newcommand{\Geo}[2]{\Gq_{\mathrm{eo}} (#1, #2)}

\newcommand{\GL}{GL}

\newcommand{\Ort}{O}
\newcommand{\Tauq}[1]{\tau_q (#1)}